\documentclass{cmslatex}
\usepackage[paperwidth=7in, paperheight=10in, margin=.875in]{geometry}
 \usepackage[backref,colorlinks,linkcolor=red,anchorcolor=green,citecolor=blue]{hyperref}
\usepackage{amsfonts,amssymb}
\usepackage{amsmath}
\usepackage{color}
\usepackage{graphicx}
\usepackage{cite}
\usepackage{enumerate}
\renewcommand{\theequation}{\arabic{section}.\arabic{equation}}

\usepackage{amsmath,amsfonts,amssymb,amsxtra,caption,graphicx,mathrsfs,subfigure}

\newcommand{\email}[1]{\emph{E-mail:} \href{mailto:{#1}}{#1}}
\newcommand{\R}{{\mathbb R}}

\newcommand{\be}[1]{\begin{equation}\label{#1}}
\newcommand{\ee}{\end{equation}}
\renewcommand{\(}{\left(}
\renewcommand{\)}{\right)}

\newcommand{\irh}[1]{\int_{\R^d}{#1}\,dv}

\newcommand{\irk}[1]{\int_{\R}{#1}\,dv}

\newcommand\uu{{\mathbf u}}
\newcommand\uuu{{\mathbf 0}}
\newcommand\vv{{\mathbf v}}
\newcommand{\scalar}[2]{\langle{#1},{#2}\rangle}

\allowdisplaybreaks
\begin{document}

\title{PHASE TRANSITION AND ASYMPTOTIC BEHAVIOUR OF
FLOCKING CUCKER-SMALE MODEL}

\author{Xingyu Li\thanks{ CEREMADE (CNRS UMR n$^\circ$ 7534), PSL university, Universit\'e Paris-Dauphine, Place de Lattre de Tassigny, 75775 Paris 16, France. \email{li@ceremade.dauphine.fr} 
}}\pagestyle{myheadings}\markboth{Phase transition and asymptotic behaviour}{Xingyu Li}
\maketitle\thispagestyle{empty}

\begin{abstract} In this paper, we study a continuous flocking Cucker-Smale model with noise, which
has isotropic and polarized stationary solutions depending on the intensity of the noise. The first result
establishes the threshold value of the noise parameter which drives the phase transition. This threshold
value is used to classify all stationary solutions and their linear stability properties. Using an entropy,
these stability properties are extended to the non-linear regime. The second result is concerned with
the asymptotic behaviour of the solutions of the evolution problem. In several cases, we prove that
stable solutions attract the other solutions with an optimal exponential rate of convergence determined
by the spectral gap of the linearized problem around the stable solutions. The spectral gap has to be
computed in a norm adapted to the non-local term.
\end{abstract}

\begin{keywords} Flocking model; phase transition; symmetry breaking; stability; large time asymptotics; free energy; spectral gap; asymptotic rate of convergence\end{keywords}
\begin{AMS} 35B40; 35P15; 35Q92.\end{AMS}

\section{Introduction}

In many fields such as biology, ecology or economic studies, emerging collective behaviours and self-organization in multiagent interactions have attracted the attention of many researchers. In this paper we consider the McKean-Vlasov model in order to describe \emph{flocking} {\color{blue}(see \cite{D,F} for more information)}. The original model of~\cite{MR2324245} is Cucker-Smale model, which describes a population of $N$ birds moving in $\R^3$ by the equations
\[
v_i(t_n+\Delta t)-v_i(t_n)=\frac{\lambda\,\Delta t}{N}\sum_{i=1}^{N}a_{ij}\big(v_j(t_n)-v_i(t_n)\big)\,,\quad i=1\,,\;2\ldots\;N
\]
at discrete times $t_n=n\Delta t$ with $n\in\mathbb{N}$ and $\Delta t>0$. Here $v_i$ is the velocity of the $i$th bird, the model is homogeneous in the sense that there is no position variable, and the coefficients $a_{ij}$ model the interaction between pairs of birds as a function of their relative velocities, while $\lambda$ is an overall coupling parameter. The authors proved that under certain conditions on the parameters, the solution converges to a state in which all birds fly with the same velocity. Another model is the Vicsek model~\cite{BF,G,PhysRevLett.75.1226} which was derived earlier to study the evolution of a population in which individuals have a given speed but the direction of their velocity evolves according to a diffusion equation with a local alignment term. This model exhibits phase transitions. In~\cite{MR3067586,MR3305654,MR2914250,MR3180036}, phase transition has been shown in a continuous version of the model: with high noise, the system is disordered and the average velocity is zero, while for low noise a direction is selected. 

Here we consider a model on $\R^d$, $d\ge 1$ with noise as in~\cite{MR2401690,MR2860672}. {In \cite{MR2860672}, the authors mainly study the noise at the level of individual-based models, which is about the collective behavior of animals. And in \cite{MR2401690}, the authors mainly study the model that every individual adjusts its velocity by adding to it a weighted average of the differences of its velocity with those of the other birds, and it is Brownian white noise. In our model,} the population is described by a distribution function $f(v,t)$ in which the interaction occurs through a mean-field nonlinearity known as \emph{local velocity consensus} and we also equip the individuals with a so-called \emph{self-propulsion} mechanism which privileges a speed (without a privileged direction) but does not impose a single value to the speed as in the Vicsek model. The distribution function solves
\be{fl}
\frac{\partial f}{\partial t}=D\,\Delta f+\nabla\cdot\Big((v-\uu_f)\,f+\alpha\,v\(|v|^2-1\)f\Big)\,,\quad f(.,0)=f_{\rm{in}}>0
\ee
where $t\ge 0$ denotes the time variable and $v\in\R^d$ is the velocity variable. Here $\nabla$ and $\Delta$ are the gradient and the Laplacian with respect to $v$ respectively. The parameter $D>0$ measures the intensity of the noise, $\alpha>0$ is the parameter of self-propulsion which tends to force the distribution to be centered on velocities $|v|$ of the order of $1$ when $\alpha$ becomes large, and
\[
\uu_f(t)=\frac{\irh{v\,f(t,v)}}{\irh{f(t,v)}}
\]
is the \emph{mean velocity}. We refer to~\cite{MR3199779} for more details. Notice that~\eqref{fl} is one-homogeneous: from now on, we will assume that the mass satisfies $\irh{f(t,v)}=1$ for any $t\ge0$, without loss of generality. In~\eqref{fl}, the velocity consensus term $v-\uu_f$ can be interpreted as a friction force which tends to align $v$ and $\uu_f$. Altogether, individuals are driven to a velocity corresponding to a speed of order $1$ and a direction given by~$\uu_f$, but this mechanism is balanced by the noise which pushes the system towards an isotropic distribution with zero average velocity. The Vicsek model can be obtained as a limit case in which we let $\alpha\to+\infty$: see~\cite{MR3109433}. The competition between the two mechanisms, relaxation towards a non-zero average velocity and noise, is responsible for a phase transition between an ordered state for small values of $D$, with a distribution function $f$ centered around $\uu$ with $\uu\neq\uuu$, and a disordered, \emph{symmetric} state with $\uu=\uuu$. This phase transition can also be interpreted as a \emph{symmetry breaking} mechanism from the \emph{isotropic distribution} to an ordered, asymmetric or \emph{polarized distribution}, with the remarkable feature that nothing but the initial datum determines the direction of $\uu_f$ for large values of $t$ and any stationary solution generates a continuum of stationary solutions by rotation. We refer to~\cite{MR3180036} for more detailed comments and additional references on related models.

So far, the flocking behaviour has been studied in \cite{CFR} for a Boltzmann-type equation, which describes the large time behaviour by grazing collision limit. 
In~\cite{MR3180036}, J.Tuguat proved the result of phase transitions for a particular class of McKean-Vlasov diffusions, where the dimensioin $d=1$ and the potential is more general. And recently in \cite{C}, the authors studied phase transitions for Mckean-Vlasov equation on the torus. Moreover, it has been proved in~\cite{MR3541988} by A.~Barbaro, J.~Canizo, J.~Carrillo and P.~Degond that stationary solutions are isotropic for large values of $D$ while symmetry breaking occurs as $D\to0$. The bifurcation diagram showing the phase transition has also been studied numerically in~\cite{MR3541988} and the phase diagram can be found in~\cite[Theorem~2.1]{MR3180036}.
\subsection{Main results.}The first purpose of this paper is to classify all stable and unstable stationary solutions and establish a complete description of the phase transition.  The figure of \cite{MR3541988} shows the existence of the critical
point. The proof can be found in Section \ref{Sec:Thm1.1completed}.

\begin{thm}\label{Thm:Main1} Let $d\ge 1$ and $\alpha>0$. There exists a critical intensity of the noise $D_*>0$ such that
\begin{enumerate}
\item[(i)] if $D\ge D_*$ there exists one and only one non-negative stationary distribution which is isotropic and stable,

 \item[(iI)]if $d=1$ and $D <D_*$, there exist one and only one non-negative isotropic station-
ary distribution which is unstable, and two stable non-negative non-symmetric stationary distributions, and the two non-symmetric stationary solutions are
the same after the reflection about the vertical axis.
\item[(iIi)] if $d\ge 2$ and $D<D_*$, there exist one and only one non-negative isotropic stationary distribution which is instable, and a continuum of stable non-negative non-symmetric stationary distributions, but this non-symmetric stationary solution is unique up to a rotation.
\end{enumerate}\end{thm}

The stability of the stationary solutions is about the critical point of the free energy
defined in \eqref{free}. Under the assumption of mass normalization to $1$, it is straightforward to observe that any stationary solution can be written as
\be{stsol}
f_\uu(v)=\frac{e^{-\frac1D\,\(\frac12\,|v-\uu|^2+\tfrac\alpha4\,|v|^4-\tfrac\alpha2\,|v|^2\)}}{\irh{e^{-\frac1D\,\(\frac12\,|v-\uu|^2+\tfrac\alpha4\,|v|^4-\tfrac\alpha2\,|v|^2\)}}}
\ee
where $\uu=(u_1,..u_d)\in\R^d$ solves $\irh{(\uu-v)\,f_\uu(v)}=0$. Up to a rotation, we can assume that $\uu=(u,0,...0)=u\,e_1$ and the question of finding stationary solutions to~\eqref{fl} is reduced to solve $u\in\R$ that satisfies the equation
\be{con1}
\mathcal H(u):=\irh{(v_1-u)\,e^{-\frac1D(\phi_\alpha(v)-u\,v_1)}}=0,\quad\mbox{where}\quad\phi_\alpha(v):=\tfrac\alpha4\,|v|^4+\tfrac{1-\alpha}2\,|v|^2\,
\ee

Here $\phi_{\alpha}(v)$ is the environmental potential that the dynamics is evolving. Obviously $u=0$ is always a solution. Moreover, if $u$ is a solution of $\eqref{con1}$, then $-\,u$ is also a solution. As a consequence, from now on, we always suppose that $u\ge 0$. Taking $\mathbf u= u\,e_1$ is a key but straightforward idea in case of stationary solution, which however does not adapt so easily to non-stationary solutions of the evolution problem.  Theorem~\ref{Thm:Main1} is proved in Section~\ref{Sec:PhaseTransition} by analyzing~\eqref{con1}. We will see that for dimension $d\ge 2$,
\eqref{con1} is a rotation about the vertical axis of the one-dimensional potential, so there is a
continuum of the stationary solutions that are unique by rotation.

The second purpose of this paper is to study the stability of the stationary states and the rates of convergence of the solutions of the evolution problem. A key tool is the \emph{free energy}
\be{free}
\mathcal F[f]:=D\irh{f\,\log f}+\irh{f\,\phi_\alpha}-\frac12\,|\uu_f|^2
\ee
and we shall also consider the \emph{relative entropy} with respect to $f_\uu$ defined as
\[
\mathcal F[f]-\mathcal F[f_\uu]=D\irh{f\,\log\(\frac f{f_\uu}\)}-\frac12\,|\uu_f-\uu|^2
\]
where $f_\uu$ is a stationary solution to be determined. Notice that $f_\uu$ is a critical point of ~$\mathcal F$ under the mass constraint. We say $f_{\uu}$ is stable, if and only if fu is the critical point
of $\mathcal F$. For $D >D_*$, the only stationary solution $f_{\uuu}$ is stable, and it is the local minimizer
of $\mathcal F$. For $0<D <D_*$, the local minimizers of $\mathcal F$ are $f_{\uu}$ with $\uu\ne\uuu$, and they are stable. See Section~\ref{Sec:F} for more details.

To a distribution function $f$, we associate the \emph{non-equilibrium Gibbs state}
\be{Gibbs}
G_f(v):=\frac{e^{-\frac1D\,\(\frac12\,|v-\uu_f|^2+\tfrac\alpha4\,|v|^4-\tfrac\alpha2\,|v|^2\)}}{\irh{e^{-\frac1D\,\(\frac12\,|v-\uu_f|^2+\tfrac\alpha4\,|v|^4-\tfrac\alpha2\,|v|^2\)}}}\,.
\ee
Unless $f$ is a stationary solution of~\eqref{fl}, let us notice that $G_f$ does not solve~\eqref{fl}. A crucial observation is that
\[
\mathcal F[f]=D\irh{f\,\log f}+\frac12\irh{|v-\uu_f|^2\,f}+\irh{\(\frac\alpha4\,|v|^4-\frac\alpha2\,|v|^2\)f}
\]
is a Lyapunov function in the sense that
\[
\frac d{dt}\mathcal F[f(t,\cdot)]=-\,\mathcal I[f(t,\cdot)]
\]
if $f$ solves~\eqref{fl}, where $\mathcal I[f]$ is the \emph{relative Fisher information} of $f$ defined as
\be{fisher}
\mathcal I[f]:=\irh{\left|D\,\frac{\nabla f}f+\alpha\,v\,|v|^2+(1-\alpha)\,v-\uu_f\,\right|^2f}=D^2\irh{\left|\nabla\log\(\frac f{G_f}\)\right|^2f}\,.
\ee
It is indeed clear that $\mathcal F[f(t,\cdot)]$ is monotone non-increasing and $\frac d{dt}\mathcal F[f(t,\cdot)]=0$ if and only if $f=G_f$ is a stationary solution of~\eqref{fl}. This is consistant with our first stability result, which can be directly deduced from Lemma \ref{Lem:Stab}.
\begin{prop}\label{Prop:stability} For any $d\ge1$ and any $\alpha>0$, $f_\uuu$ is a linearly stable critical point if and only if $D>D_*$.\end{prop}

Actually, from the dynamical point of view, we have a better, global result.
\begin{thm}\label{Thm:Main2} For any $d\ge1$ and any $\alpha>0$, if $D>D_*$, then for any solution $f$ of~\eqref{fl} with nonnegative initial datum $f_{\rm{in}}$ of mass $1$ such that $\mathcal F[f_{\rm{in}}]<\infty$, there are two positive constants $C$ and $\lambda$ such that, for any time $t>0$,
\be{ExpCV}
0\le\mathcal F[f(t,\cdot)]-\mathcal F[f_\uuu]\le C\,e^{-\lambda\,t}\,.
\ee
\end{thm}
For general $D >0$ and radially symmetric initial data, \eqref{ExpCV} is proved in Proposition \ref{Prop:ExpCV}.
When $D >D_*$ and the initial data is more general, the proof can be found in Section \ref{Sec:Isotropic}.
Moreover, we shall also prove that
\[
\irh{|f(t,\cdot)-f_\uuu|^2\,f_\uuu^{-1}}\le C\,e^{-\lambda\,t}
\]
with same $\lambda>0$ as in Theorem~\ref{Thm:Main2}, but eventually for a different value of $C$, and characterize $\lambda$ as the spectral gap of the linearized evolution operator in an appropriate norm. A characterization of the optimal rate $\lambda$ is given in~Theorem~\ref{Thm:Main2Refined}.

Moreover, for general $D>0$, besides the convergence of the free energy, we also have
the result about continuity and convergence of the norm of the average velocity, which
is Proposition \ref{Prop:uf}. The similar technique has been also used in Section 2.1 of \cite{BF}. See
Section \ref{sec46} for more details.

For $D<D_*$, the situation is more subtle as it hits the difficult question of the \emph{bassins
of attraction}. The solution of~\eqref{fl} can in principle converge either to the \emph{isotropic stationary solution} $f_\uuu$ or to a \emph{non-symmetric stationary solution} $f_\uu$ with $\uu\neq\uuu$. We will prove that  for nonnegative initial datum $f_{in}$ with mass 1 that satisfies $\mathcal F[f_{in}]<\mathcal F[f_{\uuu}] $, $\mathcal F[f]-\mathcal F[f_{\uu}]$ decays with an exponential rate which is also characterized by a spectral gap in Section~\ref{Sec:Polarized}. In the non-symmetric case, the question of the rate of convergence to a solution with a uniquely defined limiting $\uu$ or a set of polarized solutions is still open.

{\subsection{Motivations and ideas of the proofs.}As we introduced before, the equation \eqref{fl} that we focus on has been studied in \cite{MR3541988} as a simplified model for a pack of birds, and the authors proved that
when $D\to 0$ the system is `ordered' and when $D\to \infty$ the system is `disordered', which means there exists a phase transition. A natural question is, do we have a more explicit result? Is there a value $D_*$ that accurately divides the `ordered' and `disordered' state? We have seen in \cite{MR3541988} that, the system is `ordered' or not depends on the solutions of $\mathcal H(u)$ defined in \eqref{con1}. If \eqref{con1} has no postive solutions, then the sytem is `ordered', otherwise it is `disordered'. Answering the question means giving more accurate results about $\mathcal H(u)$.\\

By analyzing Fig.~\ref{Fig1}, we can find out that when $D$ is large, then $\mathcal H'(0)<0$, and $\mathcal H(u)$ is strictly decreasing on $(0, \infty)$, so it does not have positive solutions. When $D$ is small, then $\mathcal H'(0)>0$, and $\mathcal H(u)$ is first increasing and then decreasing on $(0,\infty)$, which means that $\mathcal H(u)$ has \emph{one and only one} postive solution on $(0,\infty)$. So we can conjecture that there exists $D_*$ that satisfies $\mathcal H'(0)=0$, such that $\mathcal H(u)=0$ has no positive solutions when $D>D_*$, and $\mathcal H(u)=0$ has \emph{one and only one} postive solution when $0<D<D_*$. This means a phase transition exists. In Section~\ref{Sec:PhaseTransition}, we first show that such a $D_*$ exists and is unique, and has the conjecture of properties. The whole proof is basic, just about mathematical analysis, but the case $d\ge 2$ is more complicated.

The next natural question is, can we get more information about large time asymptotics of the system? \cite{MR3541988} has explained that phase transition is directly related to the different possibilities of the stationary solutions defined in \eqref{stsol} as $t\to\infty$. From the results of Section~\ref{Sec:PhaseTransition}, we know that when $D>D_*$, the stationary solution $f_{\uuu}$ is unique, and we should naturally guess that the solution of \eqref{fl} converges to $f_{\uuu}$ in some way, and moreover, we hope that the convergence has an exponential decay.

Our idea goes similarly as in \cite{MR3196188}. We know from  Section \ref{Sec:F} that if $D>D_*$, then $f_{\uuu}$ is the unique mimizer of the free energy $\mathcal F$ and we can also prove the convergence. We study the time deriviative of $\mathcal F$, which is the Fisher information $\mathcal I$. Around $f_{\uuu}$, we consider the quadratic forms associated with the expansion of $\mathcal F$ and $\mathcal I$, and we prove a coercivity result relating the two quadratic forms $Q_1$ and $Q_2$, which is based on a Poincar\'e inequality, see Section~\ref{Sec:Linearization}. The goal of studying the quadratic forms is to prove the result of large time asymptotic behaviour. For the equation \eqref{fl}, we consider the linearzied equation with the linear operator $\mathcal L$. After choosing an adopted Hilbert space with scalar product $\langle, \rangle$, such that $\langle g,g\rangle=Q_1[g]$  and $\langle g,\mathcal L g\rangle=-Q_2[g]$. The coercivity constant $\mathcal P_D$ is the spectral gap of $\mathcal L$. For the non-linear part of \eqref{fl}, we can prove that the difference with the nonlinear expansion is small enough compared to $\langle g,g\rangle$  and $\langle g,\mathcal L g\rangle$, so the result of large time asymptotics is proved by using Gr\"onwall's inequality. Moreover, the sharp exponential rate is just $2\mathcal P_D$. Details can be found in Section~\ref{Sec:Isotropic}.

One can change probably the potential $\phi_{\alpha}(v)$ to a polynomial of higher order in $v$ (of course it should also take negative values of $\phi(\alpha)=0$), and similar results on the phase transitions and large time asymptotic behaviour when the intensity of noise $D$ is large enough, but the phase diagram will be more complicated.

The difficulty of the case $0<D<D_*$ is that the stationary solution is not unique, and there is even a continuum of solutions for $d\ge 2$, so the convergence of the solutions of \eqref{fl} is much more complicated, which leave a number of open questions for further studies.}

\subsection{Structure of the paper.}In Section~\ref{Sec:PhaseTransition}, we classify all stationary solutions, prove Theorem~\ref{Thm:Main1} and deduce that a phase transition occurs at $D=D_*$. Section~\ref{Sec:Linearization} is devoted to the linearization. The relative entropy and the relative Fisher information provide us with two quadratic forms which are related by the linearized evolution operator. The main result here is to prove a spectral gap property for this operator in the appropriate norm, which is inspired by a similar method used in~\cite{MR3196188} to study the sub-critical Keller-Segel model: see Proposition~\ref{Prop:Coercivity}. In Section \ref{Sec:F}, we give the properties
of the free energy, study all terms in the linearization, including the term arising from the non-local mean velocity. The proof of Theorem~\ref{Thm:Main2} follows using a Gr\"onwall type estimate, in Section~\ref{Sec:Isotropic} (isotropic case). In Section~\ref{Sec:Polarized}, we also give some results in the polarized case.

\section{Stationary solutions and phase transition}\label{Sec:PhaseTransition}

The aim of this section is to classify all stationary solutions of~\eqref{fl} as a first step of the proof of the phase transition result of Theorem~\ref{Thm:Main1}.  For the case $d=1$, We refer to 
\cite{MR3180036}.

\subsection{A technical observation}

Recall that $\phi_\alpha(v):=\tfrac\alpha4\,v^4+\tfrac{1-\alpha}2\,v^2\,$ {and define $|\mathbb S^{d-1}|$ as the volume of $d$-dimensional unit sphere $\mathbb S^{d-1}$.} We start by the simple observation that
\[
-\,D\,\frac\partial{\partial v_1}\(e^{-\frac1D(\phi_\alpha(v)-u\,v_1)}\)=\(v_1-u+\alpha\(|v|^2-1\)v_1\)e^{-\frac1D(\phi_\alpha(v)-u\,v_1)}
\]
can be integrated on $\R^d$ to rewrite $\mathcal H$ as
\[
\mathcal H(u)=\alpha\irh{\(1-|v|^2\)v_1\,e^{-\frac1D(\phi_\alpha(v)-u\,v_1)}}
\]
and compute
\[
\mathcal H'(u)=\frac\alpha D\irh{\(1-|v|^2\)v_1^2\,e^{-\frac1D(\phi_\alpha(v)-u\,v_1)}}\,.
\]
We observe that $\mathcal H'(0)=\frac\alpha D\,|\mathbb S^{d-1}|\,h_d(D)$ with the convention that $|\mathbb S^0|=2$, where
\[
h_d(D):=\int_0^\infty{(s^{d+1}-s^{d+3})\,e^{-\frac{\varphi_\alpha(s)}D}\,ds}.
\]
With these notations, we are now in a position to state a key ingredient of the proof.
\begin{prop}\label{Prop:hd} For any $d\ge1$ and any $\alpha>0$, $h_d$ has a unique positive root $D_*$. Moreover $h_d$ is positive on $(0,D_*)$ and negative on $(D_*,+\infty)$.\end{prop}
\begin{proof} Our goal is to prove that $h_d=j_{d+1}-j_{d+3}$ is positive on $(0,D_*)$ and negative on $(D_*,+\infty)$ for some $D_*>0$, where
\be{def:jd}
j_d(D):=\int_0^\infty{s^d\,e^{-\frac1D\,\varphi_\alpha(s)}\,ds}\,.
\ee
Let us start with two useful identities. A completion of the square shows that for any $n\in \mathbb{N}$,
\be{square}
j_{n+5}-2\,j_{n+3}+j_{n+1}=\int_0^\infty{s^{n+1}\(s^2-1\)^2e^{-\frac{\phi_\alpha}D}\,ds}>0\,.
\ee
With an integration by parts, we obtain that
\be{ipp}
\alpha\,j_{n+5}+(1-\alpha)\,j_{n+3}=\int_0^\infty{s^{n+2}\,\varphi_\alpha'\,e^{-\frac1D\,\varphi_\alpha}\,ds}=(n+2)\,D\,j_{n+1}\,.
\ee
Next, we split the proof in a series of claims.

$\bullet$ \emph{The function $h_d$ is positive on $(0,1/(d+2)]$ and negative on $[1/d,+\infty)$}. Let us prove this claim. With $n=d$ and $n=d-2$, we deduce from~\eqref{square} and~\eqref{ipp} that
\[
h_d>\frac{1-(d+2)\,D}{1+\alpha}\,j_{d+1}\quad\mbox{and}\quad h_d<\frac{1-d\,D}{1+\alpha}\,j_{d-1}\,.
\]
As a consequence, if $h_d(D)=0$, then $D\in(1/(d+2),1/d)$.

$\bullet$ \emph{If $\alpha\le1$, then $h_d(D)=0$ has a unique solution}. By a direct computation, we observe that
\[
4\,D^2\,h_d'=\alpha\,h_{d+4}+2\,(1-\alpha)\,h_{d+2}
\]
using~\eqref{ipp} with $n=d+2$. If $\alpha\in(0,1)$, it follows that $h_d'<0$ on $[1/(d+2),+\infty)$, which proves the claim.

$\bullet$ \emph{If $\alpha>1$ and $h_d'(D_\circ)=0$ for some $D_\circ\in(1/(d+2),1/d)$, then $h_d(D_\circ)>0$}. Indeed, using
\[
4\,D^2\,h_d'=-\,\alpha\,j_{d+7}+(3\,\alpha-2)\,j_{d+5}+\,2\,(1-\alpha)\,j_{d+3}=0\,,
\]
combined with \eqref{ipp} for $n=d+2$ and $n=d$, we find that, at $D=D_\circ$,
\[
h_d(D_\circ)=\frac{(d+2)\,D-1+\alpha\,(1-d\,D)}{\alpha-1+(d+4)\,D\,\alpha}\,j_{d+1}\,.
\]

Collecting our observations concludes the proof. See Fig.~\ref{fig2} for an illustration.
\hfill\ \end{proof}

\subsection{The one-dimensional case}

\begin{lem}\label{Lem:Elementary} Let us consider a continuous positive function $\psi$ on $\R^+$ such that the function $s\mapsto\psi(s)\,e^{s^2}$ is integrable and define
\[
H(u):=\int_0^{+\infty}\(1-s^2\)\,\psi(s)\,\sinh(s\,u)\,ds\quad\forall\,u\ge0\,.
\]
For any $u>0$, $H''(u)<0$ if $H(u)\le0$. As a consequence, $H$ changes sign at most once on $(0,+\infty)$. 
\end{lem}

As a consequence, there exists a non-zero solution to \eqref{con1} if and only if $\mathcal H'(0)>0$.

\begin{proof} We first observe that
\be{RelatConcavity}
H''(u)-H(u)=\int_0^{+\infty}\(1-s^2\)\(s^2-1\)\psi(s)\,\sinh(s\,u)\,ds<0\quad\forall\,u>0\,.
\ee
Let $u_*>0$ be such that $H(u_*)=0$. If $H'(u_*)<0$, there is a neighborhood of $(u_*)_+$ such that both $H$ and $H'$ are negative. As a consequence, by continuation, $H'(u)<H'(u_*)<0$ for any $u>u_*$. We also get that $H'(u)<0$ for any $u>u_*$ if $H'(u_*)=0$ because we know that $H''(u_*)<0$. We conclude by observing that $H'(u_*)>0$ would imply $H'(u)>H'(u_*)$ for any $u\in(0,u_*)$, a contradiction with $H(0)=0$.\hfill\ \end{proof}

\begin{prop}\label{Prop:D1} Assume that $d=1$ and $\alpha>0$. With the notations of Proposition~\ref{Prop:hd}, Equation~\eqref{con1}, \emph{i.e.}, $\mathcal H(u)=0$, has as a solution $u=u(D)>0$ if and only if $D<D_*$ and $\lim_{D\to(D_*)_-}u(D)=0$.\end{prop}

As a consequence, if $d=1$ and $D<D_*$, the there are exactly 3 stationary solutions. This result has already been established in \cite{MR3180036}, and we give a short proof for completion.

\begin{proof} Since $\mathcal H(0)=0$, for any $D\neq D_*$, $h_d(D)$ and $\mathcal H(u)$ have the same sign in a neighborhood of $u=0_+$. Next we notice that
\[
-\,\frac1\alpha\,\mathcal H(u)=\int_0^\infty{\(v^2-1\)v\,e^{-\frac{\phi_\alpha(v)}D}\,e^{\frac{u\,v}D}\,dv}-\int_0^\infty{\(v^2-1\)v\,e^{-\frac{\phi_\alpha(v)}D}\,e^{-\frac{u\,v}D}\,dv}\,.
\]
The second term of the right-hand side converges to $0$ as $u\to\infty$ by the dominated convergence theorem. Concerning the first term, let us notice that $|(v^2-1)\,v|\,e^{-\phi_\alpha(v)/D}$ is bounded on $(0,3)$, so that
\begin{multline*}
\int_0^\infty{\(v^2-1\)v\,e^{-\frac{\phi_\alpha(v)}D}\,e^{\frac{u\,v}D}\,dv}\\
\ge\int_0^1{\(v^2-1\)v\,e^{-\frac{\phi_\alpha(v)}D}\,e^{\frac{u\,v}D}\,dv}+\int_2^3{\(v^2-1\)v\,e^{-\frac{\phi_\alpha(v)}D}\,e^{\frac{u\,v}D}\,dv}\\
\ge -\,C_1\,e^{u/D}+C_2\,e^{2\,u/D}\to+\infty\quad\mbox{as}\quad u\to+\infty
\end{multline*}
for some positive constants $C_1$ and $C_2$. This proves that $\lim_{u\to+\infty}\mathcal H(u)=-\infty$ and shows the existence of at least one positive solution of~\eqref{con1} if $h_d(D)>0$.

The fact that~\eqref{con1} has at most one solution on $(0,+\infty)$ follows from Lemma~\ref{Lem:Elementary} applied with $H(u)=\mathcal H(D\,u)$ and $\psi(v)=2\,\alpha\,v\,e^{-\frac{\phi_\alpha(v)}D}$. Finally, as consequence of the regularity of $H$ and of~\eqref{RelatConcavity}, the solution $u=u(D)$ of~\eqref{con1} is such that $\lim_{D\to(D_*)_-}u(D)=0$.

For $D=D_*$, notice that $\mathcal H'(0)=\mathcal H''(0)=0$, and $\mathcal H'''(0)=\frac{\alpha}{D^2}\irk{(1-v^2)v^3e^{-\frac{\phi_*}{D_*}}}<0$ because 
\[
-\irk{(1-v^2)v^3e^{-\frac{\phi_*}{D_*}}}=\irk{v(1-v^2)^2e^{-\frac{\phi_*}{D_*}}}>0
\]
By the same method as above, we conclude that $\mathcal H(u)$ has no positive solutions.
\hfill\ \end{proof}

\subsection{The case of a dimension $d\ge2$}

We extend the result of Proposition~\ref{Prop:D1} to higher dimensions. In radial coordinates such that $s=|v|$ and $v_1=s\,\cos\theta$, with $\theta\in[0,\pi]$,
\[
\mathcal H(u)=\alpha\,\left|\mathbb S^{d-2}\right|\int_0^\pi\int_0^{+\infty}\(1-s^2\)\,s^d\,e^{-\frac{\varphi_\alpha(s)}D}\,\cos\theta\,(\sin\theta)^{d-2}\,e^{\frac{u\,s}D\cos\theta}\,ds\,d\theta
\]
it can also be rewritten as
\[
\mathcal H(u)=2\,\alpha\,\left|\mathbb S^{d-2}\right|\int_0^{\pi/2}\int_0^{+\infty}\(1-s^2\)\,s^d\,e^{-\frac{\varphi_\alpha(s)}D}\,\cos\theta\,(\sin\theta)^{d-2}\,\sinh\(\tfrac{u\,s}D\,\cos\theta\)\,ds\,d\theta\,.
\]
\begin{prop}\label{Prop:D2} Assume that $d\ge2$ and $\alpha>0$. With the notations of Proposition~\ref{Prop:hd}, Equation~\eqref{con1}, \emph{i.e.}, $\mathcal H(u)=0$, has as a solution $u=u(D)>0$ if and only if $D<D_*$ and $\lim_{D\to(D_*)_-}u(D)=0$.\end{prop}

Qualitatively, the result is the same as in dimension $d=1$: there exists a solution to~\eqref{con1} if and only if $\mathcal H'(0)>0$. See Fig.~\ref{Fig1}.
As a conclusion, for $d\ge 2$, there is a rotation
about the vertical axis of the one-dimensional potential, and we have a continuum of
stationary solutions which can be determined by rotation once one is found. Notice that now Lemma~\ref{Lem:Elementary} does not apply directly. We consider
\be{mathsfh}
\mathsf h(s):=\int_0^{\pi/2}\cos\theta\,(\sin\theta)^{d-2}\,\sinh(s\,\cos\theta)\,d\theta\,.
\ee
We need the lemma below.
\begin{lem}\label{Prop2.3tmp} Assume that $d\ge2$. The function $\mathsf h$ defined by~\eqref{mathsfh} is such that $s\mapsto s\,\mathsf h'(s)/\mathsf h(s)$ is monotone increasing on $(0,+\infty)$.\end{lem}

\begin{proof} Let $s_1$ and $s_2$ be such that $0<s_1<s_2$ and consider a series expansion. With
\[
P_n:=\int_0^\pi(\cos\theta)^{2n}\,(\sin\theta)^{d-2}\,d\theta\,,
\]
we know that
\begin{align*}
&s_2\,\mathsf h'(s_2)\,\mathsf h(s_1)=\sum_{m=0}^\infty\frac{s_2^{2m+1}}{(2m)!}\,P_{m+1}\;\sum_{n=0}^\infty\frac{s_1^{2n+1}}{(2n+1)!}\,P_{n+1}\,,\\
&s_1\,\mathsf h'(s_1)\,\mathsf h(s_2)=\sum_{m=0}^\infty\frac{s_1^{2m+1}}{(2m)!}\,P_{m+1}\;\sum_{n=0}^\infty\frac{s_2^{2n+1}}{(2n+1)!}\,P_{n+1}\,.
\end{align*}
These series are absolutely converging and we can reindex the difference of the two terms using $i=\min\{m,n\}$ to get
\begin{multline*}
s_2\,\mathsf h'(s_2)\,\mathsf h(s_1)-s_1\,\mathsf h'(s_1)\,\mathsf h(s_2)\\=\sum_{i=0}^\infty\sum_{j=1}^\infty\frac{(s_1\,s_2)^{2i+1}}{(2i+2\,j+1)!\,(2i+1)!}\,P_{i+1}\,P_{j+1}\,\frac{2i+2\,j+1}{2\,(i+j+1)}\(s_2^j-s_1^j\)>0\,.
\end{multline*}
\hfill\ \end{proof}

Now we come back to the proof of  Proposition~\ref{Prop:D2}.

\noindent\emph{Proof of Proposition~\ref{Prop:D2}.} We prove that $\lim_{u\to+\infty}\mathcal H(u)=-\infty$ as in the case $d=1$ by considering the domains defined in the coordinates $(s,\theta)$ by $0\le s\le1$ and $\theta\in[0,\pi/2]$ on the one hand, and $2\le s\le3$ and $0\le\theta\le\theta_*$ for some $\theta_*\in(0,\pi/6)$ on the other hand. 

If $D\ge D_*$, we obtain from $\mathcal H'(0)\le 0$ that
\[
\int_0^1{\(1-s^2\)\,s^{d+1}\,e^{-\frac{\varphi_\alpha(s)}D}\,ds}\le\int_1^\infty{\(s^2-1\)\,s^{d+1}\,e^{-\frac{\varphi_\alpha(s)}D}\,ds}
\]
obviously $\mathsf h'(s)$ is strictly increasing on $(0,\infty)$, which means that for any $u>0$,
\[
\begin{aligned}
\int_0^1{\(1-s^2\)\,s^{d+1}\,e^{-\frac{\varphi_\alpha(s)}D}\mathsf h'\left(\frac{us}{D}\right)\,ds}&<\int_0^1{\(1-s^2\)\,s^{d+1}\,e^{-\frac{\varphi_\alpha(s)}D}\mathsf h'\left(\frac{u}{D}\right)\,ds}\\
&\hspace*{12pt}=\int_1^\infty{\(s^2-1\)\,s^{d+1}\,e^{-\frac{\varphi_\alpha(s)}D}\mathsf h'\left(\frac{u}{D}\right)\,ds}\\
&\hspace*{12pt}<\int_1^\infty{\(s^2-1\)\,s^{d+1}\,e^{-\frac{\varphi_\alpha(s)}D}\mathsf h'\left(\frac{us}{D}\right)\,ds}\\
\end{aligned}
\]
so $\mathcal H'(u)<0$ for any $u>0$, which proves that $\mathcal H(u)$ has no positive solutions when $D\ge D_*$.

For $D<D_*$, the existence of at least one solution $u>0$ of $\mathcal H(u)=0$ follows from Proposition~\ref{Prop:D1}. If there exist $0<u_1<u_2$ such that $\mathcal H(u_1)=\mathcal H(u_2)=0$, then 
\[
\int_0^1{\(1-s^2\)\,s^d\,e^{-\frac{\varphi_\alpha(s)}D}\,\mathsf h(\tilde u_1\,s)\,ds}=\int_1^\infty{\(s^2-1\)\,s^d\,e^{-\frac{\varphi_\alpha(s)}D}\,\mathsf h(\tilde u_1\,s)\,ds}
\]
where $\tilde u_1:=u_1/D<u_2/D=:\tilde u_2$. We deduce from Lemma~\ref{Prop2.3tmp} that the function $s\mapsto\mathsf k(s):=\mathsf h(\tilde u_2\,s)/\mathsf h(\tilde u_1\,s)$ is a monotone increasing function on $(0,+\infty)$. Using $\mathcal H(u_1)=0$, we obtain
\[
\begin{aligned}
\int_0^1{\(1-s^2\)\,s^d\,e^{-\frac{\varphi_\alpha(s)}D}\,\mathsf h(\tilde u_2\,s)\,ds}&=
\int_0^1{\(1-s^2\)\,s^d\,e^{-\frac{\varphi_\alpha(s)}D}\,\mathsf h(\tilde u_1\,s)\,\mathsf k(s)\,ds}\\
&<\int_0^1{\(1-s^2\)\,s^d\,e^{-\frac{\varphi_\alpha(s)}D}\,\mathsf h(\tilde u_1\,s)\,\mathsf k(1)\,ds}\\
&\hspace*{12pt}=\int_1^\infty{\(s^2-1\)\,s^d\,e^{-\frac{\varphi_\alpha(s)}D}\,\mathsf h(\tilde u_1\,s)\,\mathsf k(1)\,ds}\\
&\hspace*{12pt}<\int_1^\infty{\(s^2-1\)\,s^d\,e^{-\frac{\varphi_\alpha(s)}D}\,\mathsf h(\tilde u_1\,s)\,\mathsf k(s)\,ds}\\
&\hspace*{24pt}=\int_1^\infty{\(s^2-1\)\,s^d\,e^{-\frac{\varphi_\alpha(s)}D}\,\mathsf h(\tilde u_2\,s)\,ds}\,,
\end{aligned}
\]
a contradiction with $\mathcal H(u_2)=0$.\hfill\ $\square$

\subsection{Classification of the stationary solutions and phase transition}

We learn form the expression of $\mathcal I$ in~\eqref{fisher} that any stationary solution of~\eqref{fl} is of the form $f_\uu$ with $\uu=u\,e_1$ for some $u$ which solves~\eqref{con1} up to an rotation. Since $\mathcal H(0)=0$, $u=0$ is always a solution. According to Propositions~\ref{Prop:D1} and~\ref{Prop:D2}, Equation~\eqref{con1} has a solution $u=u(D)$ if and only if $D>D_*$ where $D_*$ is obtained as the unique positive root of $h_d$ by Proposition~\ref{Prop:hd}.
\begin{cor}\label{Cor:classification} Let $d\ge 2$ and $\alpha>0$. With the above notations and $D_*$ defined as in Proposition~\ref{Prop:hd}, we know that
\begin{enumerate}
\item[(i)] if $D\ge D_*$ there exists one and only one non-negative stationary distribution $f_\uu$ given by $\uu=\uuu$, which is isotropic,
\item[(ii)]if $d=1$ and $D <D_*$ there exists one and only one non-negative isotropic stationary distribution with $\uu=\uuu$, and two stable non-negative, non-symmetric stationary distributions $f_{\uu}$ with $\uu=\pm u(D)$.

\item[(iii)] if $d\ge 2$ and $D<D_*$ there exists one and only one non-negative isotropic stationary distribution with $\uu=\uuu$, and a continuum of stable non-negative non-symmetric stationary distributions $f_\uu$ with $\uu=u(D)\,\mathsf e$ for any $\mathsf e\in\mathbb S^{d-1}$, with the convention that $\mathbb S^0=\{-1,1\}$.
\end{enumerate}
There are no other stationary solutions.\end{cor}
In other words, we have obtained the complete classification of the stationary solutions of~\eqref{fl}, which shows that there are two phases of stationary solutions: the isotropic one with $\uu=\uuu$, and only for $D<D_*$, there are the non-isotropic ones with $\uu\neq\uuu$,   which are either the same after the resection about the
vertical axis for $d=1$, or unique up to a rotation for $d\ge 2$. To complete the proof of Theorem~\ref{Thm:Main1}, we have to study the linear stability of these stationary solutions, and the proof can be found in Section \ref{Sec:Thm1.1completed}.

\subsection{An important estimate}\label{Sec:Important}

The next result is a technical estimate which is going to play a key role in our analysis.
\begin{lem}\label{Lem:Important} Assume that $d\ge1$, $\alpha>0$ and $D>0$.\begin{enumerate}
\item[(i)] In the case $\uu=\uuu$, we have that $\irh{|v|^2\,f_\uuu}>d\,D$ if and only if $D<D_*$.
\item[(ii)] In the case $D\in(0,D_*)$ and $\uu\neq\uuu$, we have that
\[
\irh{\big|(v-\uu)\cdot\uu\big|^2\,f_\uu}<D\,|\uu|^2\,.
\]
\item[(iii)] In the case $d\ge2$ and $D\in(0,D_*)$ and $\uu\neq\uuu$, we have that
\[
\irh{\big|(v-\uu)\cdot\mathbf w\big|^2\,f_\uu}=D\,|\mathbf w|^2\quad\forall\,w\in\R^d\quad\mbox{such that}\quad\uu\cdot\mathbf w=0\,.
\]

\end{enumerate}\end{lem}
\begin{proof} Using Definition~\eqref{def:jd}, we observe that $\irh{|v|^2\,f_\uuu}-d\,D$ has the sign of
\[
j_{d+1}-d\,D\,j_{d-1}=\alpha\,\big(j_{d+1}-j_{d+3}\big)=\alpha\,h_d(D)
\]
by~\eqref{ipp} with $n=d-2$. This proves \emph{(i)} according to Proposition~\ref{Prop:hd} and Corollary~\ref{Cor:classification}.

By integrating $D\,\uu\cdot\nabla\big((\uu\cdot v)\,f_\uu\big)$, we obtain that
\begin{multline*}
0=\irh{\Big(D\,|\uu|^2-(\uu\cdot v)^2\(\alpha\,|v|^2+1-\alpha\)+u\,(\uu\cdot v)\Big)\,f_\uu}\\
=D\,|\uu|^2-
\irh{\big|(v-\uu)\cdot\uu\big|^2\,f_\uu}+D\,|\uu|^2\,\mathcal H'(|\uu|)
\end{multline*}
Then \emph{(ii)} follows from Propositions~\ref{Prop:D1} and~\ref{Prop:D2} because $\mathcal H'(u)<0$ if $u=u(D)=|\uu|$.

With no loss of generality, we can assume that $\uu=(u,0,\ldots0)\neq\uuu$. By integrating $\frac\partial{\partial v_1}f_\uu$ on $\R^d$, we know that $\irh{\(|v|^2-1\)v_1\,f_\uu}=0$. Let us consider radial coordinates such that $s=|v|$ and $v_1=s\,\cos\theta$, with $\theta\in[0,\pi]$. From the integration by parts
\[
(d-1)\,D\int_0^\pi\cos\theta\,(\sin\theta)^{d-2}\,e^{\frac{u\,s}D\cos\theta}\,d\theta=u\,s\int_0^\pi(\sin\theta)^d\,e^{\frac{u\,s}D\cos\theta}\,d\theta\,,
\]
we deduce that $\irh{\(|v|^2-1\)\(1-v_1^2\)\,f_\uu}=0$ because $s^2\,(\sin\theta)^2=1-v_1^2$ and
\[
\irh{\(|v|^2-1\)\,v_i^2\,f_\uu}=0\quad\forall\,i\ge2
\]
by symmetry among the variables $v_2$, $v_3$,\ldots $v_d$. We conclude by integrating $\frac\partial{\partial v_i}f_\uu$ on~$\R^d$ that
\[
\irh{|v_i|^2\,f_\uu}=D\quad\forall\,i\ge2\,,
\]
which concludes the proof of \emph{(iii)}.\hfill\ \end{proof}

\begin{cor}\label{Cor:Important} Assume that $d\ge1$, $\alpha>0$ and $\mathsf e\in\mathbb S^{d-1}$. There exists a function $D\mapsto\kappa(D)$ on $(0,D_*)$ which is continuous with values in $(0,1)$ such that, with $\uu=u(D)\,\mathsf e$,
\[
\frac1D\irh{\big|(v-\uu)\cdot\mathbf w\big|^2\,f_\uu}=\kappa(D)\,(\mathbf w\cdot\mathbf e)^2+|\mathbf w|^2-(\mathbf w\cdot\mathbf e)^2\quad\forall\,\mathbf w\in\R^d\,.
\]
\end{cor}
With $\kappa(D):=\frac1{u(D)^2}\irh{\big|(v-\uu)\cdot\uu\big|^2\,f_\uu}$ for an arbitrary $\mathsf e\in\mathbb S^{d-1}$, the proof is a straightforward consequence of Lemma~\ref{Lem:Important}.

\subsection{The additional result of $u(D)$}\label{Sec:Qualitative result}

The main goal of this subsection is to show the qualitative result of $u(D)$ when $D<D_*$ and $D\to(D_*)_{-}$.
\begin{prop}\label{Prop:D3}
Let $0<D<D_*$. If $u(D)$ denotes the positive solution of $\mathcal H(u)=0$, then

\[
\lim_{D\to D_*}\frac{(u(D))^2}{D_*-D}=\frac{\alpha((1-\alpha)(1-dD_*)-2D_*)}{1-(d+2)D_*}.
\]

\end{prop}
\begin{proof}
According to implicit function theorem, when $D\to (D_*)_{-}$, $u(D)$ is a differentiable function of $D$, and
\[
\frac{\partial \mathcal H}{\partial u}\frac{\partial u}{\partial D}=-\frac{\partial \mathcal H}{\partial D}=-\frac{1}{D^2}\irh{(v_1-u)(\phi_{\alpha}-v_1u)e^{-\frac{1}{D}(\phi_{\alpha}-uv_1)}}.
\]
At $D=D_*$, $\frac{\partial \mathcal H}{\partial u}=0$, and  $\frac{\partial^2 \mathcal H}{\partial u^2}=\frac{1}{D^2}\irh{(1-|v|^2)v_1^3e^{-\frac{\phi_{\alpha}}{D_*}}}=0$, so we obtain that
\[
\frac{\partial \mathcal H}{\partial u}\sim \beta u^2
\]
as $D\to (D_*)_{-}$, where
\[
\beta=\frac{1}{2}\frac{\partial^3 \mathcal H}{\partial u^3}(0)=\frac{1}{2D_*^3}\irh{(1-|v|^2)v_1^4e^{-\frac{\phi_{\alpha}}{D_*}}}<0.
\]
On the other hand, from integrating by parts and the indentity $\irh{v_1^2e^{-\frac{\phi_{\alpha}}{D_*}}}=D_*\irh{e^{-\frac{\phi_\alpha}{D_*}}}$ deduced from Lemma~\ref{Lem:Important}, we obtain
\[
\irh{(v_1-u)(\phi_{\alpha}-v_1u)e^{-\frac{1}{D}(\phi_{\alpha}-uv_1)}}=\frac{1}{4}(\alpha-4)u\irh{v_1^2e^{-\frac{1}{D}(v_1^2\phi_{\alpha}-uv_1)}}
\]
\[
+\frac{1}{4}\left[(2D+1-\alpha-dD+\alpha dD-\alpha D)u+3u^3\right]\irh{e^{-\frac{1}{D}(\phi_{\alpha}-uv_1)}}
\]
so that
\[
\irh{(v_1-u)(\phi_{\alpha}-v_1u)e^{-\frac{1}{D}(\phi_{\alpha}-uv_1)}}\sim\frac{1}{4}((1-\alpha)(1-dD_*)-2D_*)u\irh{e^{-\frac{\phi_{\alpha}}{D_*}}}
\]
as $D\to (D_*)_{-}$. Notice that $(1-\alpha)(1-dD_*)-2D_*<0$ because $\frac{1}{d+2}<D_*<\frac{1}{d}$. 
By using ~\eqref{ipp} and $\irh{(1-|v|^2)v_1^2e^{-\frac{\phi_{\alpha}}{D_*}}}=0$, we obtain that
\[
\frac{\irh{e^{-\frac{\phi_{\alpha}}{D_*}}}}{\irh{(1-|v|^2)v_1^4e^{-\frac{\phi_{\alpha}}{D_*}}}}=\frac{\alpha}{D_*(1-(d+2)D_*)}
\]
which concludes the proof using $\lim_{D\to D_*}\frac{(u(D))^2}{D_*-D}=-2\lim_{D\to D_*}u\frac{\partial u}{\partial D}$.\hfill\end{proof}


\section{The linearized problem: local properties of the stationary solutions}\label{Sec:Linearization}

This section is devoted to the quadratic forms associated with the expansion of the free energy $\mathcal F$ and the Fisher information $\mathcal I$ around the stationary solution $f_\uu$ studied in Section~\ref{Sec:PhaseTransition}. These quadratic forms are defined for a smooth perturbation $g$ of $f_\uu$ such that $\irh{g\,f_\uu}=0$ by
\begin{equation*}
\begin{aligned}
Q_{1,\uu}[g]:&=\lim_{\varepsilon\to 0}\frac2{\varepsilon^2}\,\left(\mathcal F\big[f_\uu(1+\varepsilon\,g)\big]-\mathcal F\big[f_\uu\big]\right)\\
&=D\irh{g^2\,f_\uu}-D^2\,|\vv_g|^2\quad\mbox{where}\;\vv_g:=\frac1D\irh{v\,g\,f_\uu}\,,
\end{aligned}
\end{equation*}
\[
Q_{2,\uu}[g]:=\lim_{\varepsilon\to 0}\frac1{\varepsilon^2}\,\mathcal I\big[f_\uu\,(1+\varepsilon\,g)\big]=D^2\irh{\left|{\nabla g-\vv_g}\right|^2\,f_\uu}\,.
\]

\subsection{Stability of the isotropic stationary solution}\label{Sec:StabIso}

The first result is concerned with the linear stability of $\mathcal F$ around $f_\uuu$.
\begin{lem}\label{Lem:Stab} On the space of the functions $g\in\mathrm L^2(f_\uuu\,dv)$ such that $\irh{g\,f_\uuu}=0$, $Q_{1,\uuu}$ is a nonnegative (resp.~positive) quadratic form if and only if $D\ge D_*$ (resp.~$D>D_*$). Moreover, for any $D>D_*$, let $\eta(D):=\alpha\,\mathcal C\,|h_d(D)$ for some explicit $\mathcal C=\mathcal C(D)>0$. Then
\be{lowerequivnrm}
Q_{1,\uuu}[g]\ge\eta(D)\,\irh{g^2\,f_\uuu}\quad\forall\,g\in\mathrm L^2(f_\uuu\,dv)\quad\mbox{such that}\quad\irh{g\,f_\uuu}=0\,.
\ee\end{lem}
\begin{proof} On one hand, if $D<D_*$, let $\mathsf e\in\mathbb S^{d-1}$. We consider $g(v)=v\cdot\mathsf e$ and, using~\eqref{ipp} with $n=d-2$, compute
\[
Q_{1,\uuu}[g]=D\irh{v_1^2\,f_\uuu}-\(\irh{v_1^2\,f_\uuu}\)^2=\,\mathcal C\int_0^\infty\(d\,D\,s^{d-1}-\,s^{d+1}\)\,e^{-\frac{\varphi_\alpha(s)}D}\,ds
\]
where the last equality determines the value of $\mathcal C$. This proves that $Q_{1,\uuu}[g]=-\,\alpha\,\mathcal C\,h_d(D)<0$. So the necessary condition for the linear stability of $f_\uuu$ is $D\ge D_*$. 

On the other hand, let $g$ be a function in $\mathrm L^2(\R^d,f_\uuu\,dv)$ such that $\irh{g^2\,f_\uuu}=\irh{v_1^2\,f_\uuu}$. We can indeed normalize $g$ with no loss of generality. With $v_1=v\cdot\mathsf e$, $\mathsf e\in\mathbb S^{d-1}$ such that $\uu_{g\,f_\uuu}=u\,\mathsf e$ for some $u\in\R$, we know by the Cauchy-Schwarz inequality that
\[
\(\irh{v_1\,g\,f_\uuu}\)^2\le\irh{g^2\,f_\uuu}\irh{v_1^2\,f_\uuu}=\(\irh{v_1^2\,f_\uuu}\)^2=\(\frac1d\irh{|v|^2\,f_\uuu}\)^2\,,
\]
hence
\[
Q_{1,\uuu}[g]\ge D\irh{v_1^2\,f_\uuu}-\(\irh{v_1^2\,f_\uuu}\)^2=-\,\alpha\,\mathcal C\,h_d(D)\,.
\]
This proves the linear stability of $f_\uuu$ if $D>D_*$.\hfill\ \end{proof}

Proposition \ref{Prop:stability} and the classification result of Theorem~\ref{Thm:Main1} is a consequence of Corollary~\ref{Cor:classification} and Lemma~\ref{Lem:Stab}.

\subsection{A coercivity result}\label{Sec:Coercivity}

From the result of \cite{P,V},  we start by recalling the \emph{Poincar\'e inequality}
\be{Poincare}
\irh{|\nabla h|^2\,f_\uu}\ge\Lambda_D\irh{|h|^2\,f_\uu}\quad\forall\,h\in\mathrm H^1\(\R^d,f_\uu\,dv\)\quad\mbox{such that}\quad\irh{h\,f_\uu}=0\,.
\ee
Here $\uu$ is an admissible velocity such that $\uu=\uuu$ if $D\ge D_*$, or $|\uu|=u(D)$ if $D<D_*$, and $\Lambda_D$ denotes the corresponding optimal constant. Since $\varphi_\alpha$ can be seen as a uniformly strictly convex potential perturbed by a bounded perturbation, it follows from the \emph{carr\'e du champ} method and the Holley-Stroock lemma that $\Lambda_D$ is a positive constant (see \cite{BE,H,J} for more details).
Let
\begin{align*}
&\uu[f]=0\quad\mbox{if}\quad D\ge D_*\quad\mbox{or}\quad\uu_f=\uuu\quad\mbox{and}\quad D<D_*\,,\\
&\uu[f]=\frac{u(D)}{|\uu_f|}\,\uu_f\quad\mbox{if}\quad D<D_*\quad\mbox{and}\quad\uu_f\neq\uuu\,.
\end{align*}
Based on~\eqref{Poincare}, we have the following coercivity result.
\begin{prop}\label{Prop:Coercivity} Let $d\ge 1$, $\alpha>0$, $D>0$ and $\mathcal C_D=D\,\Lambda_D$ with $\Lambda_D$ as in~\eqref{Poincare}, and $D_*$ as in Corollary \ref{Cor:classification}. Consider a nonnegative distribution function $f\in\mathrm L^1(\R^d)$ with $\irh f=1$, and let $\uu\in\R^d$ be such that either $\uu=\uuu$ for $D>D_*$, or $\uu\ne\uuu, |\uu|=u(D)$ if $D<D_*$. We assume that $g:=(f-f_\uu)/f_\uu\in\mathrm H^1\(\R^d,f_\uu\,dv\)$. 
\begin{enumerate}
\item[(i)] If $D>D_*$, then

\[
Q_{2,\uu}[g]\ge\mathcal P_D\,Q_{1,\uu}[g],\quad\mbox{where}\quad \mathcal P(D):=\mathcal C(D)\left(1-\frac{\int_{\mathbb R^d}|v|^2f_{\uuu}}{dD}\right).
\]

\item[(ii)]  If $0<D<D_*$, then

\[
Q_{2,\uu}[g]\ge\mathcal C_D\,\big(1-\kappa(D)\big)\,\frac{(\vv_g\cdot\uu)^2}{|\vv_g|^2\,|\uu|^2}\,Q_{1,\uu}[g]
\]

with $\vv_g:=\frac1D\irh{(v-\uu)\,g\,f_\uu}$ and $\kappa(D)<1$ defined as in Corollary~\ref{Cor:Important}.\par As a special case, if $\uu=\uu[f]$, then $Q_{2,\uu}[g]\ge\mathcal C_D\,\big(1-\kappa(D)\big)\,Q_{1,\uu}[g]$.

\end{enumerate}
\end{prop}

By construction, $\vv_g$ is such that $D\,\vv_g=\irh{(v-\uu)\,g\,f_\uu}=\irh{v\,g\,f_\uu}=\uu_f-\uu$ because $\irh{g\,f_\uu}=0$.
\begin{proof} Let us apply~\eqref{Poincare} to $h(v)=g(v)-(v-\uu)\cdot\vv_g$. Using $\vv_g=\frac1D\irh{(v-\uu)\,g\,f_\uu}$ 
and $\irh{g\,f_\uu}=0$, we obtain
\begin{multline*}
\frac1{D^2}\,Q_{2,\uu}[g]=\irh{|\nabla g-\vv_g|^2\,f_\uu}\\
\ge\Lambda_D\irh{\(g^2+|\vv_g\cdot(v-\uu)|^2-2\,\vv_g\cdot (v-\uu)\,g\)f_\uu}\\
=\Lambda_D\left[\irh{|g|^2\,f_\uu}+\irh{|\vv_g\cdot(v-\uu)|^2\,f_\uu}-2\,D\,|\vv_g|^2\right]\,.
\end{multline*}

If $D>D_*$ and $\uu=\uuu$, either $\vv_g=\uuu$ and the result is proved, or we know that $\frac1d\irh{|v|^2\,f_\uuu}< D$ by Lemma~\ref{Lem:Important}. We obtain from part (i) of Lemma \ref{Lem:Stab} that
\[
\frac1d\irh{|v|^2\,f_\uuu}\irh{|g|^2\,f_\uuu}\ge |\vv_g|^2
\]
by using the inequality: for $a,b>0, \mu>1$, if $a\ge\mu b$, then
\[
a-(2-1/\mu)b\ge (1-1/\mu)(a-b)
\]
we obtain that
\[
\irh{|g|^2\,f_\uuu}+|\vv_g|^2\(\frac1d\irh{|v|^2\,f_\uuu}-2\,D\)\ge\frac1D\left(1-\frac{\irh{|v|^2\,f_\uuu}}{dD}\right)\,Q_{1,\uuu}[g]\,,
\]
which proves the result.

If $D<D_*$ and $\uu\neq\uuu$, let us apply Corollary~\ref{Cor:Important} with $\mathbf w=\vv_g$ and $\kappa=\kappa(D)$:
\[
\irh{|\vv_g\cdot(v-\uu)|^2\,f_\uu}=\mathcal K\,D\,|\vv_g|^2\quad\mbox{with}\quad\mathcal K=1-(1-\kappa)\,\frac{(\vv_g\cdot\uu)^2}{|\vv_g|^2\,|\uu|^2}\,.
\]
We deduce from the Cauchy-Schwarz inequality
\[
D^2\,|\vv_g|^4=\(\irh{\vv_g\cdot(v-\uu)\,f_\uu}\)^2\le\irh{|g|^2\,f_\uuu}\irh{|\vv_g\cdot(v-\uu)|^2\,f_\uu}
\]
that $D\,|\vv_g|^2\le\mathcal K\irh{|g|^2\,f_\uuu}$. Hence, if $\beta\in(0,1)$, we obtain
\[
\frac1{D^2}\,Q_{2,\uu}[g]-\frac\beta{D^2}\,Q_{2,\uu}[g]\ge\big(1-\beta-(2-\mathcal K-\beta)\,\mathcal K\big)\irh{|g|^2\,f_\uuu}\,.
\]
With $\beta=1-\mathcal K$, we obtain $1-\beta-(2-\mathcal K-\beta)\,\mathcal K=0$, which proves the result.\hfill\ \end{proof}

\section{Properties of the free energy and consequences}\label{Sec:F}

We consider the free energy $\mathcal F$ and the Fisher information $\mathcal I$ defined respectively by~\eqref{free} and~\eqref{fisher}.
\newpage

\subsection{Basic properties of the free energy}\label{Sec:Fbasic}

\begin{prop}\label{Prop:Fdecay} Assume that $f_{\rm{in}}$ is a nonnegative function in $\mathrm L^1(\R^d)$ such that $\mathcal F[f_{\rm{in}}]<\infty$. Then there exists a solution $f\in C^0\(\R^+,\mathrm L^1(\R^d)\)$ of~\eqref{fl} with initial datum $f_{\rm{in}}$ such that $\mathcal F[f(t,.)]$ is nonincreasing and a.e.~differentiable on $[0,\infty)$. Furthermore
\[
\frac d{dt}\mathcal F[f(t,.)]\le-\,\mathcal I[f(t,.)]\,,\quad t>0\hspace*{12pt}\mbox{a.e.}
\]\end{prop}
This result is classical and we shall skip its proof: see for instance~\cite[Proposition~2.1]{MR2053570} for further details. One of the difficulties in the study of $\mathcal F$ is that in~\eqref{free}, the term~$|\uu_f|^2$ has a negative coefficient, so that the functional $\mathcal F$ is not convex. A smooth solution realizes the equality, and by approximations, we obtain the result.
\begin{prop}\label{Prop:Fbound} $\mathcal F$ is bounded from below on the set
\[
\left\{f\in\mathrm L_+^1(\R^d)\,:\,\irh f=1\;\mbox{and}\;\irh{|v|^4\,f}<\infty\right\}
\]
and
\[
\irh{|v|^4\,f}\le\frac1{\alpha^2}\(D+\alpha+\sqrt{(D+\alpha)^2+4\,\alpha\(\mathcal F[f]+\tfrac d2\log(2\pi)\,D\)}\hspace*{4pt}\)^2\,.
\]
\end{prop}

\begin{proof} Let $g=f/\mu$ where $\mu(v):=(2\pi)^{-d/2}\,e^{-\frac12\,|v|^2}$ and $d\mu=\mu\,dv$. Since $g\,\log g\ge g-1$ and $\int_{\R^d}(g-1)\,d\mu=0$, we have the classical estimate
\[
\irh{f\,\log f}+\frac12\irh{|v|^2\,f}=\int_{\R^d}g\(\log g-\frac d2\log(2\pi)\)d\mu\ge-\,\frac d2\log(2\pi)\,.
\]
By the Cauchy-Schwarz inequality, $|\uu|^2\le\irh{|v|^2\,f}$ and $\irh{|v|^2\,f}\le\sqrt{\irh{|v|^4\,f}}$, and we deduce that
\[
\mathcal F[f]\ge-\,\frac d2\log(2\pi)\,D+\frac\alpha4\,X^2-\frac{D+\alpha}2\,X\quad\mbox{with}\quad X:=\sqrt{\irh{|v|^4\,f}}\,.
\]
A minimization of the r.h.s.~with respect to $X>0$ shows that $\mathcal F[f]\ge-\frac{(D+\alpha)^2}{4\,\alpha}-\,\frac d2\log(2\pi)\,D$ while the inequality provides the bound on $X$.\hfill\ \end{proof}

\subsection{The minimizers of the free energy}\label{Sec:Fminimizers}

\begin{cor}\label{Cor:FreeEnergyMinimization} Let $d\ge 1$ and $\alpha>0$. The free energy $\mathcal F$ as defined by~\eqref{free} has a unique nonnegative minimizer with unit mass, $f_\uuu$, if $D\ge D_*$. Otherwise, if $D<D_*$, we have
\[
\min\mathcal F[f]=\mathcal F[f_\uu]<\mathcal F[f_\uuu]
\]
for any $u\in\R^d$ such that $|\uu|=u(D)$. The above minimum is taken on all nonnegative functions in $\mathrm L^1\big(\R^d,(1+|v|^4)\,dv\big)$ such that $\int_{\R^d}f\,dv=1$.\end{cor}
\begin{proof} Any minimizing sequence convergence is relatively compact in $\mathrm L^1\big(\R^d,\,dv\big)$ by the Dunford-Pettis theorem, $f\mapsto\uu_f$ is relatively compact and the existence of a minimizer follows by lower semi-continuity.

Next we prove that $\mathcal F[f_{\uu}]<\mathcal F[f_{\uuu}]$ for $0<D<D_{*}$ with the case $\uu=ue_1$. It equals to
\[
D\int_{\mathbb R^d}{f_{\uu}\log\left(\frac{f_{\uu}}{f_{\uuu}}\right)dv}<\frac{u^2}{2}
\]
which is 
\[
\int_{\mathbb R^d}{e^{-\frac{1}{D}(\phi_{\alpha}-uv_1+\frac{u^2}{2})}dv}>\int_{\mathbb R^d}{e^{-\frac{\phi_{\alpha}}{D}}dv}
\]
set function $\Omega(x):=\int_{\mathbb R^d}{e^{-\frac{1}{D}(\phi_{\alpha}-xv_1+\frac{x^2}{2})}dv}$. Then by Proposition \ref{Prop:D2},
\[
\Omega'(x)=\frac{e^{-\frac{x^2}{2D}}}{D}\int_{\mathbb R^d}{(v_1-x)e^{-\frac{1}{D}}(\phi_{\alpha}-xv_1)dv}>0\quad\mbox{for}\quad x\in(0,u)
\]
this means that
\[
\int_{\mathbb R^d}{e^{-\frac{1}{D}(\phi_{\alpha}-uv_1+\frac{u^2}{2})}dv}=\Omega(u)>\Omega(0)=\int_{\mathbb R^d}{e^{-\frac{\phi_{\alpha}}{D}}dv}
\]
which finishes the proof of $\mathcal F[f_{\uu}]<\mathcal F[f_{\uuu}]$.
\hfill\ \end{proof}

\subsection{Proof of Theorem~\texorpdfstring{\ref{Thm:Main1}}{1.1}}\label{Sec:Thm1.1completed}

By Corollary~\ref{Cor:FreeEnergyMinimization}, $f_\uuu$ is the unique minimizer if and only if $D\ge D_*$. It is moreover linearly stable by Lemma~\ref{Lem:Stab}. Otherwise $f_\uu$ with $|\uu|=u(D)$ is a minimizer of $\mathcal F$ and it is unique up to a rotation. Combined with the results of Corollary~\ref{Cor:classification}, this completes the proof of Theorem~\ref{Thm:Main1}.\hfill\ $\square$

\subsection{Stability of the polarized stationary solution}\label{Sec:StabPol}

Another interesting consequence of Corollary~\ref{Cor:FreeEnergyMinimization} is the linear stability of $\mathcal F$ around $f_\uu$ when $D<D_*$.
\begin{lem}\label{Lem:StabPol} Let $D\in(0,D_*)$ and $\uu\in\R^d$ such that $|\uu|=u(D)$. On the space of the functions $g\in\mathrm L^2(f_\uu\,dv)$ such that $\irh{g\,f_\uu}=0$, $Q_{1,\uu}$ is a nonnegative quadratic form.\end{lem}

The proof is straightforward as, in the range $D<D_*$, $f_\uuu$ is not a minimizer of $\mathcal F$ and the minimum of $\mathcal F$ is achieved by any $f_\uu$ with $|\uu|=u(D)$. Details are left to the reader.

\subsection{An exponential rate of convergence for radially symmetric solutions}

\begin{prop}\label{Prop:ExpCV} Let $\alpha>0$, $D>0$ and consider a solution $f\in C^0\(\R^+,\mathrm L^1(\R^d)\)$ of~\eqref{fl} with radially symmetric initial datum $f_{\rm{in}}\in\mathrm L_+^1(\R^d)$ such that $\mathcal F[f_{\rm{in}}]<\infty$. Then~\eqref{ExpCV} holds for some $\lambda>0$.\end{prop}
\begin{proof} According to Proposition~\ref{Prop:Fdecay}, we know that
\[
\frac d{dt}\big(\mathcal F[f(t,\cdot)]-\mathcal F[f_\uuu]\big)\le-\,\mathcal I[f(t,\cdot)]
\]
where $\mathcal I$ defined by~\eqref{fisher} and $\uu_f=\uuu$ because the radial symmetry is preserved by the evolution. We have a \emph{logarithmic Sobolev inequality}
\be{LogSob}
\irh{\left|\nabla\log\(\frac f{f_\uuu}\)\right|^2f}\ge\mathcal K_\uuu\irh{f\,\log\(\frac f{f_\uuu}\)}=\mathcal F[f]-\mathcal F[f_\uuu]
\ee
for some constant $\mathcal K_\uuu>0$. This inequality holds for the same reason as for the Poincar\'e inequality~\eqref{Poincare}: since $\varphi_\alpha$ can be seen as a uniformly strictly convex potential perturbed by a bounded perturbation, it follows from the \emph{carr\'e du champ} method and the Holley-Stroock lemma that $\mathcal K_\uuu$ is a positive constant. Hence
\[
\frac d{dt}\big(\mathcal F[f(t,\cdot)]-\mathcal F[f_\uuu]\big)\le-\frac{\mathcal K_\uuu}D\irh{f\,\log\(\frac f{f_\uuu}\)}=-\frac{\mathcal K_\uuu}D\,\big(\mathcal F[f(t,\cdot)]-\mathcal F[f_\uuu]\big)
\]
and we conclude that 
\[
\mathcal F[f(t,\cdot)]-\mathcal F[f_\uuu]\le\big(\mathcal F[f_{\rm{in}}]-\mathcal F[f_\uuu]\big)\,e^{-\lambda\,t}
\]
with $\lambda=\mathcal K_\uuu/D$. The fact that $\mathcal F[f(t,\cdot)]-\mathcal F[f_\uuu]\ge0$ is a consequence of Corollary~\ref{Cor:FreeEnergyMinimization}.\hfill\ \end{proof}

\subsection{Continuity and convergence of the norm of the velocity average}\label{sec46}

\begin{prop}\label{Prop:uf} Let $\alpha>0$, $D>0$ and consider a solution $f\in C^0\(\R^+,\mathrm L^1(\R^d)\)$ of~\eqref{fl} with initial datum $f_{\rm{in}}\in\mathrm L_+^1(\R^d)$ such that $\mathcal F[f_{\rm{in}}]<\infty$. Then $t\mapsto\uu_f(t)$ is a Lipschitz continuous function on $\R^+$ such that $\lim_{t\to+\infty}\uu_f(t)=\uuu$ if $D\ge D_*$. If $0<D<D_*$, along
any increasing sequence $(n_k)_{k\in\mathbb N}$ of integers, one can extract a subsequence, that we still
denote by $(n_k)_{k\in \mathbb N}$, such that, uniformly in $t\in(0,1)$, we obtain that $\lim_{t\to+\infty}\uu_f(t+n_k)=\uu$ with either $\uu=\uuu$ or $|\uu|=u(D)$ if $D\in(0,D_*)$.\end{prop}
\begin{proof} Using~\eqref{fl}, a straightforward computation shows that
\[
\frac{d\uu_f}{dt}=-\,\alpha\irh{v\(|v|^2-1\)f}
\]
Remind from Propositions~\ref{Prop:Fdecay} and~\ref{Prop:Fbound} that $\int_{\mathbb R^d}|v|^4f dv$ is bounded, so the right hand side is bounded by H\"older interpolations . By Proposition~\ref{Prop:Fbound} and H\"older's inequality, we also know that $\uu_f$ is bounded.

We have a \emph{logarithmic Sobolev inequality} analogous to~\eqref{LogSob} if we consider the relative entropy with respect to the \emph{non-equilibrium Gibbs state} $G_f$ defined by~\eqref{Gibbs} instead of the relative entropy with respect to $f_\uuu$: for some constant $\mathcal K>0$,
\[
\irh{\left|\nabla\log\(\frac f{G_f}\)\right|^2f}\ge\mathcal K\irh{f\,\log\(\frac f{G_f}\)}=\mathcal F[f]-\mathcal F[G_f]\,.
\]
By the \emph{Csisz\'ar-Kullback inequality}
\be{CK}
\irh{f\,\log\(\frac f{G_f}\)}\ge\frac14\,\|f-G_f\|^2_{\mathrm L^1(\R^d)}\,,
\ee
we end up with the fact that $\lim_{t\to+\infty}\int_t^{+\infty}\big(\irh{|f-G_f|}\big)^2\,ds=0$. Using H\"older's inequality
\[
\left|\irh{v\,\big(f-G_f\big)}\right|\le\(\irh{|f-G_f|}\)^{3/4}\(\irh{|v|^4\,(f+G_f)}\)^{1/4}
\]
the decay of $\mathcal F[f(t,\cdot)]$ and Proposition~\ref{Prop:Fbound}, we learn that $\lim_{t\to+\infty}\irh{v\,\big(f-G_f\big)}=0$. Let $\mathcal K(u):=\irh{e^{-\frac1D(\phi_\alpha(v)-u\,v_1)}}$. By definition of $\mathcal H$, we have that
\[
\irh{v\,\big(f-G_f\big)}=\uu_f-\irh{v\,G_f}=\irh{(\uu_f-v)\,G_f}=-\,\frac{\mathcal H(u)}{\mathcal K(u)}\,\frac{\uu_f}{|\uu_f|}\quad\mbox{with}\quad u=|\uu_f|\,.
\]
Since $\uu_f$ is bounded, $\mathcal K(u)$ is uniformly bounded by some positive constant and we deduce that
\[
\lim_{t\to+\infty}\mathcal H\big(|\uu_f|\big)=0\,.
\]
\hfill\ \end{proof}

\section{Large time asymptotic behaviour in the isotropic case}\label{Sec:Isotropic}

In this section, our main goal is to prove Theorem~\ref{Thm:Main2}. In this section, we shall assume that $D>D_*$.

\subsection{A non-local scalar product for the linearized evolution operator}

We adapt the strategy of~\cite{MR3196188} to~\eqref{fl}. With $\vv_g=\frac1D\irh{v\,g\,f_\uuu}$ as in Section~\ref{Sec:Linearization},
\be{ScalarProduct}
\scalar{g_1}{g_2}:=D\irh{g_1\,g_2\,f_\uuu}-D^2\,\vv_{g_1}\cdot\vv_{g_2}
\ee
is a scalar product on the space $\mathcal X:=\left\{g\in\mathrm L^2(f_\uuu\,dv)\,:\,\irh{g\,f_\uuu}=0\right\}$ by Lemma~\ref{Lem:Stab} because $\scalar gg=Q_{1,\uuu}[g]$. Let us recall that $f_\uuu$ depends on $D$ and, as a consequence, also $D\,\vv_g$. Equation~\eqref{fl} means
\[
\frac{\partial f}{\partial t}=\nabla\cdot\Big(D\,\nabla f+(v-\uu_f+\nabla\phi_\alpha)\,f\Big)
\]
and $D\,\nabla f_\uuu=-\,(v+\nabla\phi_\alpha)\,f_\uuu$. Hence~\eqref{fl} is rewritten in terms of $f=f_\uuu\,(1+g)$ as
\[
f_\uuu\,\frac{\partial g}{\partial t}=D\,\nabla\cdot\Big((\nabla g-\vv_g)\,f_\uuu-\vv_g\,g\,f_\uuu\Big)
\]
using $\uu_f=D\,\vv_g$, that is,
\be{fln}
\frac{\partial g}{\partial t}=\mathcal L\,g-\vv_g\cdot\Big(D\,\nabla g-\(v+\nabla\phi_\alpha\)g\Big)\quad\mbox{with}\quad\mathcal L\,g=D\,\Delta g-\(v+\nabla\phi_\alpha\)\cdot\(\nabla g-\vv_g\)
\ee
and collect some basic properties of $\mathcal X$ endowed with the scalar product $\scalar\cdot\cdot$ and $\mathcal L$ considered as an operator on $\mathcal X$.
\begin{lemma} Assume that $D>D_*$ and $\alpha>0$. Let us consider the scalar product defined by~\eqref{ScalarProduct} on $\mathcal X$. The norm $g\mapsto\sqrt{\scalar gg}$ is equivalent to the standard norm on $\mathrm L^2(f_\uuu\,dv)$ according to
\be{Equivalence}
\eta(D)\irh{g^2\,f_\uuu}\le\scalar gg\le D\irh{g^2\,f_\uuu}\quad\forall\,g\in\mathcal X\,.
\ee
Here $\eta$ is as in~\eqref{lowerequivnrm}. The linearized operator $\mathcal L$ is self-adjoint on $\mathcal X$ with the scalar product defined by~\eqref{ScalarProduct} in the sense that $\scalar{g_1}{\mathcal L\,g_2}=\scalar{\mathcal L\,g_1}{g_2}$ for any $g_1$, $g_2\in\mathcal X$, and such that
\be{Q2}
-\,\scalar g{\mathcal L\,g}=Q_{2,\uuu}[g]\,.
\ee
\end{lemma}
\begin{proof} Inequality~\eqref{Equivalence} is a straightforward consequence of Definition~\eqref{ScalarProduct} and~\eqref{lowerequivnrm}. The self-adjointness of $\mathcal L$ is a consequence of elementary computations. By starting with
\[
\mathcal L\,g_1=\Big[D\,\Delta g_1-\(v+\nabla\phi_\alpha\)\cdot\nabla g_1\Big]+\(v+\nabla\phi_\alpha\)\cdot\vv_{g_1}\,,
\]
we first observe that $\irh{\left[D\,\Delta g_1-\(v+\nabla\phi_\alpha\)\cdot\nabla g_1\right]\,g_2\,f_\uuu}=-\,D\irh{\nabla g_1\cdot \nabla g_2\,f_\uuu}$ and, as a consequence (take $g_2=v_i$ for some $i=1$, $2$\ldots $d$), $\vv_{\mathcal L\,g_1}=\vv_{g_1}-\irh{\nabla g_1\,f_\uuu}$. Hence
\[
-\,\scalar{\mathcal L\,g_1}{g_2}=D^2\irh{\(\nabla g_1-\vv_{g_1}\)\cdot\(\nabla g_2-\vv_{g_2}\)}\,,
\]
which proves the self-adjointness of $\mathcal L$ and Identity~\eqref{Q2}.\hfill\ \end{proof}

The scalar product $\scalar\cdot\cdot$ is well adapted to the linearized evolution operator in the sense that a solution of the \emph{linearized equation}
\be{lfl}
\frac{\partial g}{\partial t}=\mathcal L\,g
\ee
with initial datum $g_0\in\mathcal X$ is such that
\[
\frac12\,\frac d{dt}Q_{1,\uuu}[g]=\frac12\,\frac d{dt}\scalar gg=\scalar  g{\mathcal L\,g}=-\,Q_{2,\uuu}[g]
\]
and has exponential decay. According to Proposition~\ref{Prop:Coercivity}, we know that
\[
\scalar{g(t,\cdot)}{g(t,\cdot)}\le\scalar{g_0}{g_0}\,e^{-2\,\mathcal P_D\,t}\quad\forall\,t\ge0\,.
\]

\subsection{Proof of Theorem~\texorpdfstring{\ref{Thm:Main2}}{1.2}}\label{Proof:Main2}

Let us consider the nonlinear term and prove that a solution $g$ of~\eqref{fln} has the same asymptotic decay rate as a solution of the linearized equation~\eqref{lfl}. By rewriting~\eqref{fln}~as
\[
f_\uuu\,\frac{\partial g}{\partial t}=D\,\nabla\cdot\Big((\nabla g-\vv_g)\,f_\uuu\Big)-D\,\vv_g\cdot\nabla(g\,f_\uuu)
\]
with $f=f_\uuu\,(1+g)$ and using $\irh{g\,f_\uuu}=0$, we find that
\[
\frac12\,\frac d{dt}Q_{1,\uuu}[g]+\,Q_{2,\uuu}[g]=D^2\,\vv_g\cdot\irh{g\,(\nabla g-\vv_g)\,f_\uuu}\,.
\]
Using $\uu_f=D\,\vv_g$, by the Cauchy-Schwarz inequality and~\eqref{lowerequivnrm}, we obtain
\[
\(\irh{g\,(\nabla g-\vv_g)\,f_\uuu}\)^2\le\irh{g^2\,f_\uuu}\irh{|\nabla g-\vv_g|^2\,f_\uuu}\le \frac{Q_{1,\uuu}[g]}{\eta(D)}\,\frac{Q_{2,\uuu}[g]}{D^2}\,.
\]
After taking into account~Proposition~\ref{Prop:Coercivity}, we have
\[
\frac d{dt}Q_{1,\uuu}[g]\le-\,2\(1-|\uu_f(t)|\,\sqrt{\tfrac{\mathcal P_D}{\eta(D)}}\,\)Q_{1,\uuu}[g]\,.
\]
By Proposition~\ref{Prop:uf}, we know that $\lim_{t\to+\infty}|\uu_f(t)|=0$, which proves that
\be{ExpRate}
\limsup_{t\to+\infty}e^{2\,(\mathcal P_D-\varepsilon)\,t}\,Q_{1,\uuu}[g(t,\cdot)]<+\infty
\ee
for any $\varepsilon\in(0,\mathcal P_D)$. After observing that $f\,\log\(f/f_\uuu\)-(f-f_\uuu)\le\frac12\,(f-f_\uuu)^2/f_\uuu$, this completes the proof of Theorem~\ref{Thm:Main2} .\hfill\ $\square$

\subsection{A sharp rate of convergence}

We know from Proposition~\ref{Prop:Coercivity} that $Q_{2,\uuu}[g]\ge\mathcal P_D\,Q_{1,\uuu}[g]$ for any $g\in\mathrm H^1\(\R^d,f_\uuu\,dv\)$ such that $\irh{g\,f_\uuu}=0$. At no cost, we can assume that $\mathcal P_D$ is the optimal constant.
\begin{thm}\label{Thm:Main2Refined} For any $d\ge1$ and any $\alpha>0$, if $D>D_*$, then the result of Theorem~\ref{Thm:Main2} holds with optimal rate $\lambda=2\,\mathcal P_D$.\end{thm}
\begin{proof} We have to prove that~\ref{ExpRate} holds with $\varepsilon=0$. By definition of $\uu_f$, we have that
\[
|\uu_f|^2=\(\irh{v\,(f-f_\uuu)}\)^2\le\irh{g^2\,f_\uuu}\irh{|v|^2\,f_\uuu}
\]
where $g:=(f-f_\uuu)/f_\uuu$. This guarantees that $|\uu_f(t)|\le c\,\sqrt{\eta(D)\,\mathcal P_D}\,e^{-\lambda\,t/2}$. Then the function $y(t):=Q_{1,\uuu}[g(t,\cdot)]$ obeys to the differential inequality
\[
y'\le-\,2\,\mathcal P_D\(1-c\,e^{-\lambda\,t/2}\)\,y
\]
and we deduce as in Section~\ref{Proof:Main2} that $\limsup_{t\to+\infty}e^{2\,\mathcal P_D\,t}\,y(t)$ is finite by a Gr\"onwall estimate. This rate is optimal as shown by using test functions based on perturbations of $f_\uuu$.\hfill\ \end{proof}

\section{Large time asymptotic behaviour in the polarized case}\label{Sec:Polarized}

In this section, we shall assume that $0<D<D_*$. The situation is more delicate than in the isotropic case $D>D_*$, as several asymptotic behaviours can occur.

\subsection{Symmetric and non-symmetric stationary states}

By perturbation of $f_\uuu$, we know that the set of the functions $f$ such that $\mathcal F[f]<\mathcal F[f_\uuu]$ is non-empty. Notice that the minimum of $\mathcal F$ on radial functions is achieved by $f_\uuu$. It follows that any function $f$ such that $\mathcal F[f]<\mathcal F[f_\uuu]$ is non-radial.
\begin{lem}\label{Lem:NonRadial}  Let $\alpha>0$, $D\in (0,D_*)$, and consider $f\in C^0\big(\R^+,\mathrm L^1(\R^d)\big)$ as the solution of~\eqref{fl}. Moreover the initial datum $f_{\rm{in}}\ge0$ of mass $1$ satisfies $\mathcal F[f_{\rm{in}}]<\mathcal F[f_\uuu]$. Then $\lim_{t\to+\infty}|\uu_f(t)|=u(D)$ and $\lim_{t\to+\infty}\mathcal F[f(t,\cdot)]=\mathcal F[f_\uu]$ for any $\uu\in\R^d$ such that $|\uu|=u(D)$. Moreover, there exists a sequence $\{t_n\}_{n\in\mathbb N}$ with $\lim_{n\to +\infty} t_n=+\infty$ and some $\uu\in\mathbb R^d$ with $|\uu|=u(D)$ such that
\[
f(t_n+\cdot,\cdot)\longrightarrow f_\uu\quad\mbox{in}\quad\mathrm L^1(\R^+\times\R^d)\quad\mbox{as}\quad n\to+\infty\,.
\]\end{lem}
\begin{proof} We reconsider the proof of Proposition~\ref{Prop:uf}. Since $\uu=\uuu$ is forbidden by Proposition~\ref{Prop:Fdecay} and $t\mapsto\uu_f(t)$ is a converging Lipschitz function, there exists a unique limit~$\uu$ such that $|\uu|=u(D)$. The convergence result follows from the \emph{logarithmic Sobolev inequality} and the \emph{Csisz\'ar-Kullback inequality}~\eqref{CK}.\hfill\ \end{proof}

\subsection{An exponential rate of convergence for partially symmetric solutions}

Let us start with a simple case, which is to some extent the analogous of the case of Proposition~\ref{Prop:ExpCV} in the polarized case.
\begin{prop}\label{Prop:ExpCV2} Let $\alpha>0$, $D\in(0,D_*)$ and consider a solution $f\in C^0\(\R^+,\mathrm L^1(\R^d)\)$ of~\eqref{fl} with initial datum $f_{\rm{in}}\in\mathrm L_+^1(\R^d)$ such that $\mathcal F[f_{\rm{in}}]<\mathcal F[f_\uuu]$ and $\uu_{f_{\rm{in}}}=(u,0\ldots0)$ for some $u\neq0$. We further assume that $f_{\rm{in}}(v_1,v_2,\ldots v_{i-1},v_i,\ldots)=f_{\rm{in}}(v_1,v_2,\ldots v_{i-1},-\,v_i,\ldots)$ for any $i=2$, $3$,\ldots $d$. Then
\[
0\le\mathcal F[f(t,\cdot)]-\mathcal F[f_\uu]\le C\,e^{-\,\lambda\,t}\quad\forall\,t\ge0\,.
\]

 holds with $\lambda=\mathcal C_D\,\big(1-\kappa(D)\big)>0$, with the notations of Proposition~\ref{Prop:Coercivity}.\end{prop}

Here we assume that $f_{\rm{in}}(v_1,v_2,\ldots v_{i-1},v_i,\ldots)$ is even with respect to all coordinate of index $i\ge2$, so that $\uu[f]=\uuu$ or  $\uu[f]=(\pm\,u(D),0\ldots0)$ at any time $t\ge0$.
\begin{proof} According to Proposition~\ref{Prop:uf}, we know that $\uu_f$ is continuous. On the other hand, if $\uu_f=\uuu$, then
\[
\mathcal F[f]-\mathcal F[f_\uuu]=\irh{f\,\log\(\frac f{f_\uuu}\)}=\irh{\frac f{f_\uuu}\,\log\(\frac f{f_\uuu}\)\,f_\uuu}\ge X\,\log X_{\big|X=\irh f}=0
\]
by Jensen's inequality, a contradiction with the assumption that $\mathcal F[f_{\rm{in}}]<\mathcal F[f_\uuu]$ and Proposition~\ref{Prop:Fdecay}. Hence $\uu=\uu[f]$ is constant and we can reproduce with $Q_{1,\uu}[n]$ the proof done for  $Q_{1,\uuu}[n]$ in Section~\ref{Sec:Isotropic}.\hfill\ \end{proof}

\subsection{Convergence to a polarized stationary state}

To study the rate of convergence towards the stationary solutions $f_\uu$ with $\uu\neq\uuu$ in the range $D\in(0,D_*)$, we face a severe difficulty if $\uu_f$ converges tangentially to the set $u(D)\,\mathbb S^{d-1}$ of admissible velocities for stationary solutions. Otherwise we obtain an exponential rate of convergence as in Theorem~\ref{Thm:Main2}.
\begin{prop}\label{Prop:Cv} Assume that $d\ge2$, $\alpha>0$ and $D\in(0,D_*)$. Let us consider a solution~$f$ of~\eqref{fl} with nonnegative initial datum $f_{\rm{in}}$ of mass $1$ such that $\mathcal F[f_{\rm{in}}]<\mathcal F[f_\uuu]$ and assume that $\uu=\lim_{t\to+\infty}\uu_f(t)$ is uniquely defined. If $|(\uu_f-\uu)\cdot\uu|\ge\varepsilon\,u(D)\,|\uu_f-\uu|$ for some $\varepsilon>0$ and $t>0$ large enough, then there are two positive constants $C$, $\lambda$ and some $\uu\in\R^d$ such that
\[
\irh{|f(t,\cdot)-f_\uu|^2\,f_\uu^{-1}}\le C\,e^{-\lambda\,t}.
\]
\end{prop}
The constraint of $\uu_f-\uu$ fits the result of Proposition \ref{Prop:Coercivity}, so that the coercivity constant
has the lower bound.

\begin{proof}We adapt the setting of Section~\ref{Proof:Main2} to $g=(f-f_\uu)/f_\uu$ and get that
\[
\frac12\,\frac d{dt}Q_{1,\uu}[g]+\,Q_{2,\uu}[g]=D^2\,\vv_g\cdot\irh{g\,(\nabla g-\vv_g)\,f_\uu}\,.
\]
With $\mathsf Z(t):=\mathcal C_D\,\big(1-\kappa(D)\big)\,\frac{(\vv_g\cdot\uu)^2}{|\vv_g|^2\,|\uu|^2}$, we can rewrite Proposition~\ref{Prop:Coercivity} and the estimate of the nonlinear term as
\[
Q_{2,\uu}[g]\ge\mathsf Z(t)\,Q_{1,\uu}[g]\quad\mbox{and}\quad D^2\,\vv_g\cdot\irh{g\,(\nabla g-\vv_g)\,f_\uu}\le D\,|\vv_g|\,\frac{\sqrt{Q_{1,\uu}[g]\,Q_{2,\uu}[g]}}{\sqrt{\eta(D)}}
\]
By assumption, $\mathsf Z(t)\ge\mathcal C_D\,\big(1-\kappa(D)\big)\,\varepsilon^2$. The conclusion follows as in Section~\ref{Proof:Main2}.\hfill\ \end{proof}
\clearpage

\appendix\section{Some additional properties of \texorpdfstring{$D_*$}{D0}}\label{Appendix}

In this appendix, we collect some plots which illustrate Section~\ref{Sec:PhaseTransition} and state related qualitative properties of $D_*$.

\begin{prop} For any $\alpha>0$ and $d\ge1$, the critical value $D_*=D_*(\alpha,d)$ is monotone decreasing as a function of $d$, such that
\[
\frac1{d+2}<D_*(\alpha,d)<\frac1d
\]
with lower and upper bounds achieved respectively as $\alpha\to0_+$ and $\alpha\to+\infty$.
\end{prop}
\begin{proof} The monotonicity with respect to $d$ can be read from
\[
h_d(D)-h_{d+1}(D)=\int_0^\infty{s^{d+1}\(s^2-1\)^2\,e^{-\frac{\phi_\alpha}D}\,ds}>0\,.
\]
The lower bound is a consequence of
\[
\int_0^\infty\(s^{d+1}-s^{d+3}\)e^{-\frac12\,(d+2)\,s^2}\,ds=0\,.
\]
As for the upper bound, for any $D>0$, by considering the derivatives with respect to~$\alpha$ of $j_{d+1}$ and $j_{d-1}$ as defined in~\eqref{def:jd}, we notice that
\[
\frac{j_{d+1}}{j_{d-1}}\sim\frac{2\,j_{d+3}-j_{d+5}}{2\,j_{d+1}-j_{d+3}}\sim\frac{\frac{\alpha+1}\alpha\,j_{d+3}-\frac{d+2}\alpha\,D\,j_{d+1}}{2\,j_{d+1}-j_{d+3}}
\]
by L'H\^opital's rule as $\alpha\to+\infty$. We recall that $j_{d+1}(D)=j_{d+3}(D)$ at $D=D_*$. By letting $\alpha\to+\infty$ with $D=D_*$, we conclude that $j_{d+1}/j_{d-1}\to1$. On the other hand˜ \eqref{ipp} with $n=d-2$ means that $j_{d+1}(D_*)=d\,D_*\,j_{d-1}(D_*)$, from which we conclude that $\lim_{\alpha\to+\infty}D_*(\alpha,d)=1/d$.\hfill\ \end{proof}

\bigskip We conclude this appendix by computations of $D_*$ for specific values of the parameters.\\[4pt]
$\bullet$ If $d=1$, $\alpha=2$, $D_*$ solves $(1-4\,D)\,I_{-1/4}\(\frac1{16\,D}\)+(1+4\,D)\,I_{1/4}\(\frac1{16\,D}\)
+I_{3/4}\(\frac1{16\,D}\)+I_{5/4}\(\frac1{16\,D}\)=0$ where $I_\gamma$ denotes the modified Bessel function of the first kind. Numerically, we find that $D_*\approx 0.529$ matches~\cite[Fig.~1, p.~4]{MR3541988}.\\[4pt]
$\bullet$ If $d=2$, $\alpha=2$, we remind that $D_*\approx 0.354$: see Fig.~\ref{Fig1}.\\[4pt]
$\bullet$ If $d=2$, $\alpha=4$, $D_*\approx0.398$ solves $\(16\,\Gamma\(\frac32,\frac9{16\,D}\)-16\,\sqrt\pi\)D-8\,\Gamma\(1,\frac9{16\,D}\)\sqrt{D}+6\,\sqrt\pi-3\,\Gamma\(\frac12,\frac9{16\,D}\)=0$.\\[4pt]
For further numerical examples, we refer the reader to~\cite{MR3180036,MR3541988}.
\begin{figure}[hb]
\begin{center}
\includegraphics[width=10cm,height=8cm,clip]{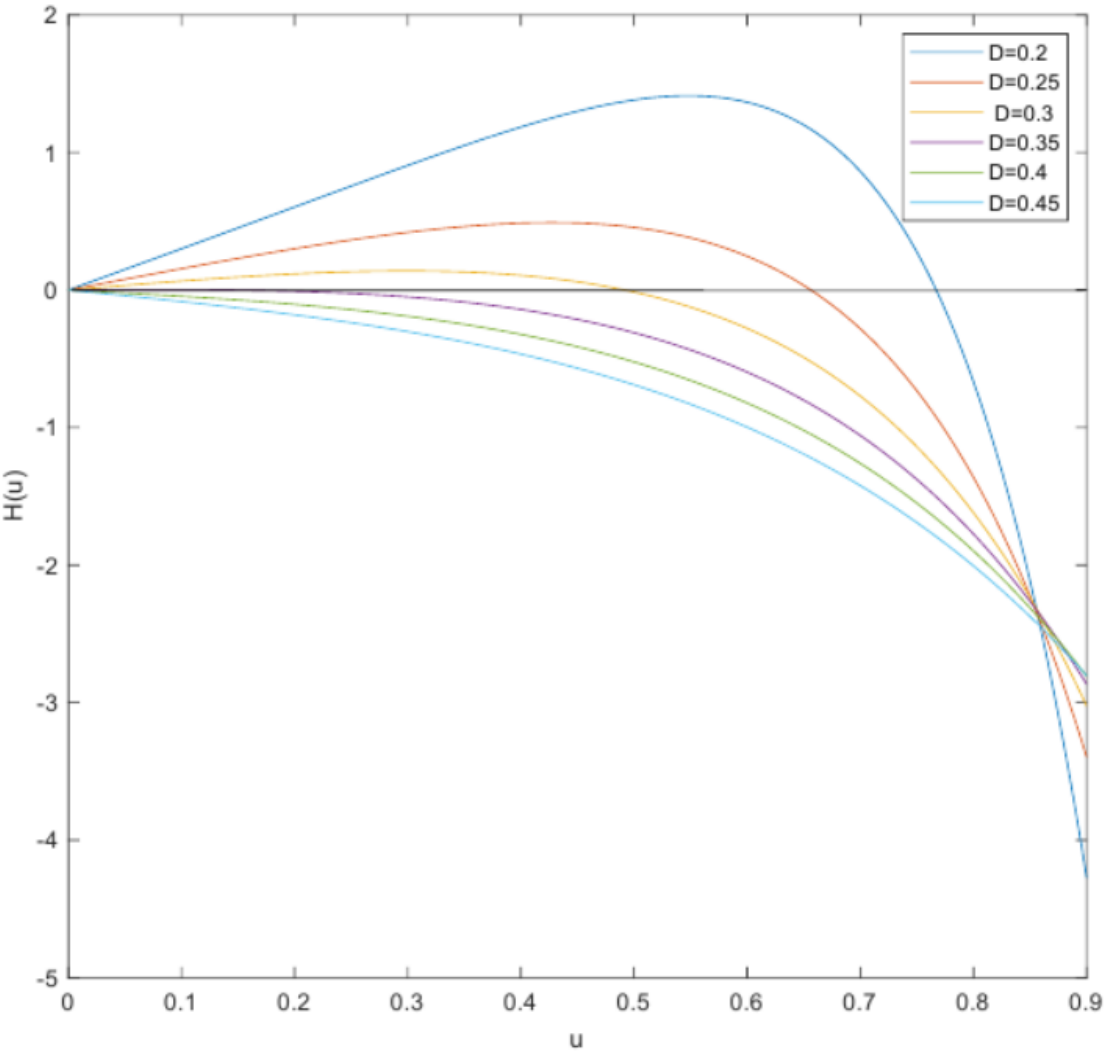}
\caption{\label{Fig1} Plot of $u\mapsto\mathcal H(u)$ when $d=2$, $\alpha=2$, and $D=0.2$, $0.25$, \ldots $0.45$. In this particular case, $D_*\approx 0.354$ solves $\(8\,\operatorname\Gamma\(\frac32,\frac1{8\,D}\)-8\,\sqrt\pi\)D
-\operatorname\Gamma\(\frac12,\frac1{8\,D}\)+2\,\sqrt\pi=0$.}
\end{center}
\end{figure}
\begin{figure}[ht]
\begin{center}
\includegraphics[width=10cm,height=8cm,clip]{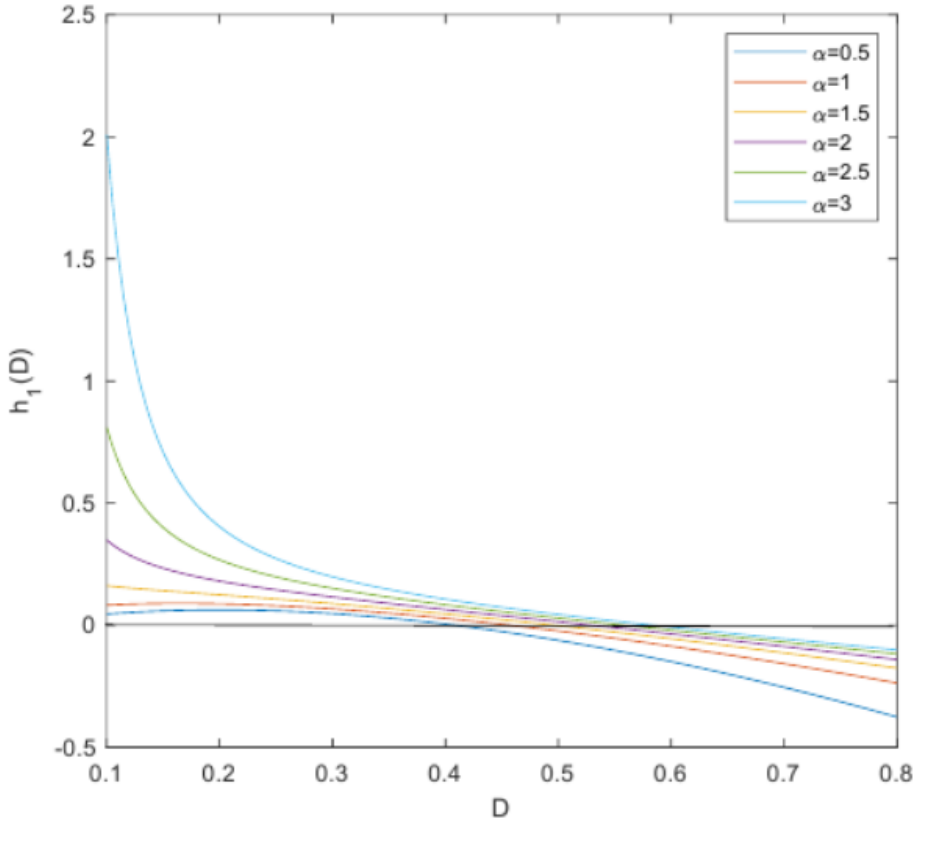}
\caption{\label{fig2} Plot of $h_d(D)$ against $D$ when $d=1$ with $\alpha=0.5$, $1$, \ldots $3.0$.}
\end{center}
\end{figure}
\\

\bigskip{\bf Acknowledgments} This work has been supported by the Project EFI ANR-17-CE40-0030 of the French National Research Agency. The author thanks J. Tugaut and A.Frouvelle for the comments and suggestions  which have led to significant improvements of the
results.\\
\noindent{\scriptsize\copyright\,2019 by the author. This paper may be reproduced, in its entirety, for non-commercial purposes.}
\newpage

\end{document}